\documentclass[11pt]{article}

\usepackage[T1]{fontenc}
\usepackage[utf8]{inputenc}
\usepackage{lmodern}
\usepackage{microtype}
\usepackage{amsmath,amssymb,amsthm,mathtools}
\usepackage{enumitem}
\usepackage[affil-it]{authblk}
\usepackage[margin=1in]{geometry}
\usepackage[colorlinks=true,linkcolor=blue,citecolor=blue,urlcolor=blue]{hyperref}

\allowdisplaybreaks
\numberwithin{equation}{section}
\setlist[itemize]{leftmargin=2em,itemsep=2pt,topsep=4pt}
\setlist[enumerate]{leftmargin=2.4em,itemsep=2pt,topsep=4pt}

\newtheorem{theorem}{Theorem}[section]
\newtheorem{conjecture}[theorem]{Conjecture}
\newtheorem{lemma}[theorem]{Lemma}

\newtheorem{corollary}[theorem]{Corollary}
\theoremstyle{remark}

\DeclareMathOperator{\cost}{cost}
\newcommand{\e}{\varepsilon}
\newcommand{\R}{\mathbb R}
\newcommand{\dotcupunion}{\mathbin{\dot\cup}}
\newcommand{\CC}{\mathcal{CC}}
\newcommand{\CCstar}{\mathcal{CC}^{*}}
\newcommand{\CD}{\mathcal{CD}}
\newcommand{\CDstar}{\mathcal{CD}^{*}}
\newcommand{\ed}{d_{\mathrm{edit}}}

\newcounter{proofcase}

\newcommand{\proofcase}[2]{
  \refstepcounter{proofcase}
  \par\medskip
  \noindent\textbf{Case \theproofcase: #2.}
  \label{#1}
  \par\smallskip
}

\usepackage[backend=biber, style=numeric, doi=false, url=false, isbn=false, maxbibnames=99]{biblatex}
\title{Clique decompositions and covers for large graphs}
\author{Wentao Zhang\thanks{Email: {\tt wt-zhang@mail.tsinghua.edu.cn}.}}
\affil{Yau Mathematical Sciences Center, Tsinghua University, Beijing, 100084, China}
\date{\today}

\begin{document}
\maketitle

\begin{abstract}
In 1966, Erd\H{o}s, Goodman, and P{\'o}sa showed that every \(n\)-vertex graph \(G\) admits a cover of its edges using at most \(\lfloor \frac{n^2}{4}\rfloor\) cliques, with tightness witnessed by the balanced complete bipartite graph.
Erd\H{o}s suggested the following strengthening: every \(n\)-vertex graph \(G\) admits an edge decomposition into cliques with total cost at most \(\lfloor \frac{n^2}{4}\rfloor\), where each \(i\)-clique has cost \(i-1\).
There is another generalization conjectured by Dau, Milenkovic and Puleo: for every \(t\ge4\), every \(n\)-vertex graph \(G\) admits a cover of its \(t\)-cliques using at most \(\prod_{j=0}^{t-1}\left\lfloor\frac{n+j}{t}\right\rfloor\) cliques.
Balogh, He, Krueger, Nguyen and Wigal proved asymptotic and fractional versions of these conjectures. We solve both conjectures for large \(n\) and identify the equality cases.
Our proof of Erd\H{o}s's conjecture combines linear programming duality and the spanning-forest polytope with a weighted stability form of Mantel's theorem, followed by explicit decomposition algorithms around an almost balanced bipartition.
For the \(t\)-clique-cover conjecture, we combine graph removal lemma and generalized Tur\'an stability with an exact covering construction for graphs close to \(T_{n,t}\).
\end{abstract}

\section{Introduction}

A \emph{\(t\)-clique} \(K_t\) is a complete graph on \(t\) vertices. A \emph{\(t\)-clique decomposition(cover)}
is a family of cliques such that every \(t\)-clique in \(G\) is a subgraph of exactly one(at least one) clique from that family.
A \emph{cost vector} \(\mathbf{c}\) is a sequence of nonnegative real numbers \(\mathbf{c}=\left(c_i\right)_{i=1}^\infty\),
each \(K_i\) would cost \(c_i\). Let \(\CD_t(G,\mathbf c)\)(\(\CC_t(G,\mathbf c)\)) denote the
minimum cost among all \(t\)-clique decompositions(covers) of \(G\) under cost vector \(\mathbf c\).
When \(\mathbf c\) is the constant \(1\) vector, we write \(\CD_t(G,\mathbf 1)=\CD_t(G)\)(\(\CC_t(G,\mathbf 1)=\CC_t(G)\)).

We study the following extremal problem: what is the maximum value of \(\CD_t(G,\mathbf c)\)(\(\CC_t(G,\mathbf c)\))
over all \(n\)-vertex graphs \(G\)? The first result on this topic was by Erd{\H o}s, Goodman, and P{\'o}sa~\cite{EGP}:

\begin{theorem}[Erd{\H o}s, Goodman, and P{\'o}sa~\cite{EGP}]\label{thm:EGP}
  If \(G\) is an \(n\)-vertex graph, then \(\CC_2(G) \le \lfloor n^2/4\rfloor\), and equality holds
  if and only if \(G\) is a balanced complete bipartite graph. Furthermore, if \(G\) is an \(n\)-vertex graph,
  then \(\CD_2(G,\mathbf{c}) \le \lfloor n^2/4\rfloor\), where \(c_2 = c_3 = 1\) and \(c_i = \infty\) for all \(i \ge 4\).
\end{theorem}

The condition \(c_i = \infty\) for all \(i \ge 4\) says we only need edges and triangles for the decomposition.
Katona and Tarj\'an conjectured the following generalization of Theorem~\ref{thm:EGP} for the cost function \(c_i=i\) for all \(i\):

\begin{theorem}[Gy\H{o}ri and Kostochka~\cite{GyoriKostochka}, Kahn~\cite{Kahnthm}, Chung~\cite{Chung}]\label{thm:GyKKC}
  Let \(c_i = i\) for all \(i\). If \(G\) is an \(n\)-vertex graph, then \(\CD_2(G, \mathbf{c}) \le 2 \cdot \lfloor n^2/4\rfloor\),
  and equality holds if and only if \(G\) is a balanced complete bipartite graph.
\end{theorem}

Theorem~\ref{thm:GyKKC} generalizes Theorem~\ref{thm:EGP} since
\begin{equation*}\CD_2(G,\mathbf 1)=\frac12\CD_2(G,2\cdot \mathbf 1)\le \frac12\CD_2(G,(i)).\end{equation*}
Later, McGuinness~\cite{McGuinness} showed that Theorem~\ref{thm:GyKKC} can be proved in a greedy manner.
Recently, \cite{BLPPPV,KLMP} determined the maximum value of \(\CD_2(G,\mathbf{c})\) for all \(\mathbf{c}\) with \(c_i = \infty\) for all \(i \geq 4\) and all sufficiently large \(n\),
with two possible extremal graphs \(K_{\lfloor n/2\rfloor,\lceil n/2\rceil}\) and \(K_n\) when \(c_2=2, c_3=3\).

Erd\H{o}s suggested (see \cite[Problem 43]{Tuza} or \cite{Gyori}) the following strengthening of Theorem~\ref{thm:GyKKC}.
\begin{conjecture}[Erd\H{o}s]\label{conjerdos}
Let \(c_i = i-1\) for all \(i\). If \(G\) is an \(n\)-vertex graph, then \(\CD_2(G,\mathbf{c}) \le \lfloor n^2/4\rfloor\).
\end{conjecture}

Indeed, Conjecture~\ref{conjerdos} generalizes Theorem~\ref{thm:GyKKC} since
\begin{equation*}
\CD_2(G,(i)) = 2\cdot\CD_2(G,(\frac i2))\le 2\cdot\CD_2(G,(i-1)).
\end{equation*}
Gy\H{o}ri and Keszegh~\cite{GyoriKeszegh} proved Conjecture~\ref{conjerdos} for \(K_4\)-free graphs.
They also proposed a stronger bound for this special case, which was later disproved by He, Ma, Wang and Zu~\cite{HeMaWangZu}.
Balogh, He, Krueger, Nguyen and Wigal~\cite{BHKNW} proved an asymptotic version of Conjecture~\ref{conjerdos} with
an \(o(n^2)\) error. We solve Conjecture~\ref{conjerdos} for all sufficiently large \(n\).
Let \(\mathfrak E_n\) be the class of graphs \(G\) on \(n\) vertices for which there is a balanced
bipartition \(V(G)=S\dotcupunion T\) such that \(G[S]=\overline K_{|S|}\) and \(G[S,T]=K_{|S|,|T|}\).
We prove the following theorem:

\begin{theorem}\label{dec:thm:main}
  Let \(c_i = i-1\) for all \(i\). There exists \(n^\ast\) such that for
  every graph \(G\) on \(n\ge n^\ast\) vertices, we have
  \begin{equation*}
  \CD_2(G,\mathbf{c}) \le \lfloor n^2/4\rfloor.
  \end{equation*}
  Moreover, equality holds if and only if \(G\in\mathfrak E_n\).
\end{theorem}

Another direction of generalizing Theorem~\ref{thm:EGP} is to maximize \(\CC_t(G,\mathbf 1)\) over
all \(n\)-vertex graphs \(G\) for \(t>2\). For \(n\ge t\), denote by \(T_{n,t}\) the \emph{Tur\'an graph},
the balanced complete \(t\)-partite graph on \(n\) vertices. Dau, Milenkovic, and Puleo~\cite{DMP} extended
Theorem~\ref{thm:EGP} to \(t=3\), proving that \(\CC_3(G) \le \CC_3(T_{n,3})\), and (when \(n \ge 3\)) equality
holds if and only if \(G\) is isomorphic to \(T_{n,3}\). They conjectured the following for larger \(t\).
\begin{conjecture}[Dau, Milenkovic, and Puleo~\cite{DMP}]\label{conj:clique_cover}
For every \(4\le t\le n\), for every \(n\)-vertex graph \(G\), we have \(\CC_t(G) \le \CC_t(T_{n,t})\),
with equality if and only if \(G \cong T_{n,t}\). 
\end{conjecture}

Chen, He and Xie~\cite{ChenHeXie} solved the case of \(t=4\). Balogh, He, Krueger, Nguyen and Wigal~\cite{BHKNW}
proved an asymptotic version of Conjecture~\ref{conj:clique_cover}.
We solve Conjecture~\ref{conj:clique_cover} for all sufficiently large \(n\).

\begin{theorem}\label{cov:thm:main}
  For every $t\ge 2$, there is $n_0(t)$ such that every graph $G$ on $n\ge n_0(t)$ vertices satisfies
\begin{equation*}
\CC_t(G)\le \CC_t(T_{n,t}).
\end{equation*}
Equality holds if and only if \(G\cong T_{n,t}\).
\end{theorem}

We give a sketch of our proof. First we introduce some notation. For a graph \(G\),
let \(\mathcal{K} (G)\) denote the set of cliques of \(G\), and \(\mathcal{K}_t (G)\) the set of \(t\)-cliques of \(G\).
Given a bipartition \(V(G)=P\dotcupunion Q\), \(e_G(P,Q)\) denotes the number of \(G\)-edges between \(P\) and \(Q\),
and \(m_G(P,Q)\) denotes the number of nonedges between \(P\) and \(Q\).
We call a clique \emph{internal} if all of its vertices are on the same side and \emph{crossing} otherwise.
Internal/crossing neighborhood and degree are defined similarly.
For simplicity of notation, we sometimes don't distinguish between an edge set and the subgraph it induces.

A fractional \(t\)-clique decomposition(cover) assigns a weight \(x_K\ge 0\) to every
\(K\in\mathcal K(G)\), so that the total weight of the cliques containing
each \(t\)-clique is exactly one(at least one). Denote the corresponding minimum cost
by \(\CDstar_t(G,\mathbf c)\)(\(\CCstar_t(G,\mathbf c)\)). We note that \cite{BHKNW} also
proved the fractional version of Conjecture~\ref{conjerdos} and Conjecture~\ref{conj:clique_cover}.

Taking \(t=2\) and \(c_i=i-1\) for all \(i\), the minimum cost \(\CDstar_2(G,(i-1))\) is the value of the linear program
\begin{equation*}
\begin{aligned}
 \CDstar_2(G,(i-1)):=\min\quad&
 \sum_{K\in\mathcal K(G)} (|K|-1)x_K,\\
 \text{subject to}\quad&
 \sum_{\substack{K\in\mathcal K(G)\\e\in E(K)}}x_K=1
       &&\text{for every }e\in E(G),\\
 &x_K\ge0
       &&\text{for every }K\in\mathcal K(G).
\end{aligned}
\end{equation*}
Its dual is
\begin{equation*}
\begin{aligned}
 \CDstar_2(G,(i-1))=\max\quad&
 \sum_{e\in E(G)}y_e,\\
 \text{subject to}\quad&
 \sum_{e\in E(K)}y_e\le |K|-1
       &&\text{for every }K\in\mathcal K(G),\\
 &y_e\in\mathbb R
       &&\text{for every }e\in E(G).
\end{aligned}
\end{equation*}
Let \(\mathbf{y}^+\) be the positive part of an optimal solution \(\mathbf{y}\), that is, \(y^+_e=\max\{y_e,0\}\),
extended by zero to the nonedges of \(G\). If \(K\) is a clique in the positive-support graph of \(\mathbf{y}^+\),
then the total weight of \(\mathbf{y}^+\) on \(E(K)\) is at most \(|K|-1\).
This is exactly the definition of the forest polytope of \(K\), namely the convex hull of the incidence vectors of
spanning forests of \(K\), with isolated vertices allowed. Starting from this observation, we prove a weighted version
of Mantel's theorem and related stability results in Section~\ref{dec:sec:fractional-stability}, which might be of
independent interest. This gives an almost balanced bipartition \(V(G)=A\dotcupunion B\) with few missing crossing edges
and one sparse side when \(\CDstar_2(G,(i-1))\) is close to optimal.

For a family \(\mathcal{D}\) of cliques, denote its total cost under \((c_i=i-1)_{i=1}^\infty\) by
\(
\cost(\mathcal{D}) := \sum_{K\in\mathcal D}(|K|-1).
\)
Given a bipartition \(V(G)=A\dotcupunion B\), an edge decomposition \(\mathcal D\) covers all crossing edges between \(A\) and \(B\).
We control \(\cost(\mathcal D)\) by comparing the cost of each clique to the number of crossing edges each clique covers.
A clique meeting the two sides in \(a,b\ge1\) vertices covers \(ab\) crossing edges and takes cost \(a+b-1\). 
Define the \emph{saving} of such a clique with respect to the bipartition \(A\dotcupunion B\) to be
\begin{equation*}
 ab-(a+b-1)=(a-1)(b-1).
\end{equation*}
Thus a clique with one vertex on one side has cost equal to its number of crossing edges,
while a \(K_4\) with two vertices on each side covers \(4\) crossing edges with cost \(3\) and saves \(1\).
And cliques completely contained in one side cover no crossing edge, which we would try to avoid.
Let \(\mathcal D_{A,B}\) be the set of crossing cliques in \(\mathcal D\) and let
\(\mathcal D_A,\mathcal D_B\) be the set of cliques in \(\mathcal D\) contained in the two sides respectively.
Denote the total saving of \(\mathcal D_{A,B}\) by
\begin{equation*}
 \Sigma(\mathcal D_{A,B})
 :=\sum_{K\in\mathcal D_{A,B}}
   (|K\cap A|-1)(|K\cap B|-1).
\end{equation*}
Then
\begin{equation}\label{dec:eq:1.2}
 \cost(\mathcal D)
 =e_G(A,B)+\cost(\mathcal D_A)+\cost(\mathcal D_B)
  -\Sigma(\mathcal D_{A,B}).
\end{equation}
We would design a greedy construction of an edge decomposition which contains few internal cliques.
We then show that this construction indeed uses cost no larger than \(\lfloor \frac{n^2}{4}\rfloor\)
and that equality holds exactly when \(G\in \mathfrak{E}_n\).
We remark that we only use cliques of size at most \(4\) in our construction.

We then turn to \(t\)-clique covers with cost vector \(\mathbf{1}\). Results from \cite{BHKNW} suggest we can delete
very few edges from \(G\) to get a subgraph \(G_0\) such that \(\CCstar_t(G_0)\) only decreases a little from \(\CC_t(G)\).
Suppose \(\CC_t(G)\) is very close to \(\CC_t(T_{n,t})\), then so is \(\CCstar_t(G_0)\).
We use the same symmetrization operation as in \cite{BHKNW} to reduce \(G_0\) to a complete multipartite graph.
Given nonadjacent vertices \(u,v\in V(G)\), we \emph{symmetrize} \(u\) to \(v\) by deleting
all edges incident to \(u\) and connecting \(u\) to all neighbors of \(v\). A \emph{clone} of a vertex \(u\) in \(G\)
is a vertex \(v\) such that \(u\neq v\), \(uv\notin E(G)\) and \(N_G(u)=N_G(v)\).
Being a clone of each other is an equivalence relation(if we consider \(u\) as a clone of itself) and
each equivalence class is independent. Given two nonadjacent clone classes \(U\) and \(V\),
we symmetrize \(U\) to \(V\) by symmetrizing each \(u\in U\) to some \(v\in V\).
Continuously taking symmetrizations would terminate with a complete multipartite graph.
We introduce a randomized symmetrization process which keeps, in expectation, the fractional \(t\)-clique cover number
nondecreasing and the number of \(t+1\)-cliques unchanged. We would show that \(|\mathcal K_{t+1}(G_0)|\) is bounded
by \(\CC_t(T_{n,t})-\CCstar_t(G_0)\), so we can delete very few copies of \(K_{t+1}\)'s to make it \(K_{t+1}\)-free.
Graph removal lemma~\cite{Removal} and a result by Ma and Qiu~\cite{MaQiu} show that
\(G\) is very close to \(T_{n,t}\) in the sense of edit distance.

Given a partition of \(V(G)\) into \(t\) disjoint parts \(V_1\dotcupunion\cdots\dotcupunion V_t\),
we call a \(t\)-set of vertices with one vertex in each part \emph{transversal} and otherwise \emph{non-transversal}.
The ideal case is when all \(K_t\)'s are transversal and we can cover each \(K_t\) with itself and the total cost
sums up to at most \(\CC_t(T_{n,t})\). However, there could be non-transversal \(K_t\)'s. For each non-transversal
\(K_t\) \(S\), we either find a transversal \(K_t\) \(T\) such that \(S\cup T\) is a clique and use it to
cover both \(S\) and \(T\), or we find a transversal \(t\)-set \(T\) which does not induce a \(K_t\) and is thus unused
in the counting. We would like different non-transversal \(K_t\)'s to use different transversals, so we need a matching
theorem between them. When \(\ed(G,T_{n,t})\) is small, we construct a partition of \(V(G)\) with each part further
partitioned into \(V_i=U_i\dotcupunion X_i\), where \(X_i\) is the set of vertices with many neighbors inside \(V_i\).
We show that the partition satisfies the following conditions: each \(x\in X_i\) has many neighbors in every other \(V_j\),
and the total number of vertices in all \(X_i\)'s is small. Under such a good partition, we show that the aforementioned
wanted matching is available and thus prove Theorem~\ref{cov:thm:main}.

The rest of the article is organized as follows. Sections~\ref{dec:sec:fractional-stability}--\ref{dec:sec:strict-stability}
are dedicated to Theorem~\ref{dec:thm:main}. Section~\ref{dec:sec:fractional-stability} finds an almost balanced bipartition
with few missing crossing edges and one side sparse when \(\CDstar_2(G,(i-1))\) is close to \(\lfloor \frac{n^2}{4}\rfloor\).
Section~\ref{dec:sec:exact-near-cut} constructs an edge decomposition under such a bipartition and proves the upper bound.
Section~\ref{dec:sec:strict-stability} deals with the equality case.
Sections~\ref{cov:sec:stability}--\ref{cov:sec:transversal matching} are dedicated to Theorem~\ref{cov:thm:main}.
Section~\ref{cov:sec:stability} shows \(G\) is close to \(T_{n,t}\) under edit distance when \(\CC_t(G)\) is close to \(\CC_t(T_{n,t})\).
Section~\ref{cov:sec:transversal matching} finds a good \(t\)-partition and a matching between
non-transversal \(K_t\)'s and transversal \(t\)-sets and proves Theorem~\ref{cov:thm:main}.

\section{Stability for fractional edge decompositions}\label{dec:sec:fractional-stability}

In this section we prove the following assertion: if the fractional edge decomposition number \(\CDstar_2(G,(i-1))\) is close to optimal,
then \(G\) admits a balanced bipartition \(V(G)=A\dotcupunion B\) such that one part is sparse and almost all \(A-B\) edges are present.

\begin{theorem}
  \label{dec:thm:fractional-stability}
For every \(\eta>0\) there are \(\delta>0\) and \(n_0\) such that the following holds. If \(G\) has order \(n\ge n_0\) and
\begin{equation*}
  \CDstar_2(G,(i-1))\ge \lfloor \frac{n^2}{4}\rfloor-\delta n^2,
\end{equation*}
then \(V(G)\) has a balanced bipartition \(A\dotcupunion B\) such that, after interchanging \(A,B\) if necessary,
\begin{equation}\label{dec:eq:2.1}
  e(G[A])+m_G(A,B)\le \eta n^2,
\end{equation}
where \(m_G(A,B)\) is the number of missing \(A\)--\(B\) pairs in \(G\).
\end{theorem}

We prove Theorem~\ref{dec:thm:fractional-stability} through a weighted form of Mantel's theorem.
An edge weighting \(\mathbf{w}\) is a map from \([n]\times[n]\) to \([0,1]\) such that
\(w_{ij}=w_{ji}\) for \(1\le i<j\le n\) and \(w_{ii}=0\) for \(1\le i\le n\).
Its \emph{positive-support graph} has edge \(ij\) exactly when \(w_{ij}\) is positive.
Note that below when we take \(S\) to have size \(3\), \eqref{dec:eq:2.2} says total edge weight on
each triangle can't exceed \(2\), this is exactly the condition for normal Mantel's theorem.
We remark that the condition of Lemma~\ref{dec:lem:weighted-mantel} cannot be weakened to only
weighted triangle free, that is, only requiring \eqref{dec:eq:2.2} for \(S\) of size \(3\).
Consider a weighted \(K_5\) with weight \(2/3\) on each edge, it satisfies the inequality in \eqref{dec:eq:2.2}
for each triangle but has total edge weight \(\frac{20}3\), which is greater than \(\lfloor\frac{25}{4}\rfloor\).

\begin{lemma}
  \label{dec:lem:weighted-mantel}
Let \(\mathbf{w}:[n]\times[n]\to[0,1]\) be an edge weighting.
Suppose that
\begin{equation}\label{dec:eq:2.2}
  \sum_{\substack{i,j\in S\\i<j}}w_{ij}\le |S|-1,\ \text{for\ every\ vertex\ set\ }S\text{\ in\ the\ positive-support\ graph\ of\ }\mathbf w \text{\ that\ induces\ a\ clique}.
\end{equation}
Then
\begin{equation*}
  \sum_{i<j}w_{ij}\le \lfloor\frac{n^2}{4}\rfloor.
\end{equation*}
\end{lemma}

\begin{proof}
Let \(A\) be the symmetric matrix with entries \(w_{ij}\). Maximize \(\mathbf{x}^{\mathsf T}A\mathbf{x}\) on
\begin{equation*}
  X:=\left\{(x_1, \cdots, x_n)|\sum_{i=1}^nx_i=1, 0\leq x_i\leq1, x_i \text{\ is a multiple of\ } \frac{1}{n}, i\in [n]\right\}
\end{equation*}
and choose a maximizer \(\widetilde{\mathbf{x}}\) with minimum support.
Then for every \(i,j\) with \(\widetilde{x}_i,\widetilde{x}_j>0\) we have \(w_{ij}>0\). 
If \(w_{ij}=0\), transferring mass between \(i\) and \(j\) preserves the objective and removes one of them from the support.

Consequently, on the support of \(\widetilde{\mathbf{x}}\), \eqref{dec:eq:2.2} holds for every vertex
subset.  Edmonds' forest-polytope theorem~\cite{Edmonds} says that a nonnegative vector \(\mathbf{z}\in\R^{E(K_r)}\)
is a convex combination of incidence vectors of forests if and only if
\(\mathbf{z}(E(S))\le |S|-1\) for every nonempty vertex set \(S\).  Hence the
restricted vector \(\widetilde{\mathbf{w}}\) on the support of \(\widetilde{\mathbf x}\) is such a convex combination.
If \(F\) is a forest and \(U,U'\) is a bipartition of \(V(F)\), then
\begin{equation*}
  \sum_{ij\in E(F)}\widetilde{x}_i\widetilde{x}_j
  \le \widetilde{\mathbf{x}}(U)\widetilde{\mathbf{x}}(U')\le\frac{1}{n^2}\lfloor\frac{n^2}{4}\rfloor.
\end{equation*}
Taking the convex combination gives \(\widetilde{\mathbf{x}}^{\mathsf T}A\widetilde{\mathbf{x}}\le\frac{2}{n^2}\lfloor\frac{n^2}{4}\rfloor\).
Now take \(\mathbf{x}=\left(\frac{1}{n}\right)_{i=1}^n\), then
\begin{equation*}
\sum_{i<j}w_{ij} = \frac{n^2}{2}\mathbf{x}^{\mathsf T}A\mathbf{x} \le \frac{n^2}{2}\widetilde{\mathbf{x}}^{\mathsf T}A\widetilde{\mathbf{x}} \le \lfloor\frac{n^2}{4}\rfloor.
\end{equation*}
\end{proof}

The next lemma says that, under the conditions of Lemma~\ref{dec:lem:weighted-mantel},
if the total weight \(W\) is close to optimal, then the weighting is almost integer.
Note that each term in \(\Phi\) below is \(0\) if \(w_{ij}\) is \(0\) or \(1\) and takes maximum at \(w_{ij}=\frac12\).

\begin{lemma}
\label{dec:lem:weight-near-integrality}
Let \(\mathbf{w}:[n]\times[n]\to[0,1]\) be an edge weighting
satisfying \emph{\eqref{dec:eq:2.2}} and put
\begin{equation*}
 W:=\sum_{i<j}w_{ij}=\left(\frac14-\Delta\right)n^2,\qquad
 \Phi:=\sum_{i<j}w_{ij}(1-w_{ij}).
\end{equation*}
Here \(\Delta\ge0\) by Lemma~\ref{dec:lem:weighted-mantel}. Then
\begin{equation}\label{dec:eq:assertion}
 \Phi\le7\sqrt{\Delta}\,n^2.
\end{equation}
\end{lemma}

\begin{proof}
We prove the second by induction on the number of equivalence classes defined as follows.
Declare \(u\sim v\) if
\(u=v\), or if
\begin{equation*}
 w_{uv}=0,\qquad w_{ux}=w_{vx}\quad
 \text{for every }x\notin\{u,v\}.
\end{equation*}
This is an equivalence relation. Each class has weight zero internally.
For any two distinct classes \(C,D\), the value \(w_{uv}\) is the same for all \(u\in C\) and \(v\in D\).

Suppose two distinct classes \(C,D\) have weight zero between them. Fix arbitrary \(c\in C\) and \(d\in D\).
Define \(\mathbf{w}^{D\to C}\) by retaining the weights on pairs outside \(C\cup D\),
assigning weight zero to every pair inside \(C\cup D\), and setting
\begin{equation*}
 w^{D\to C}_{ux}=w^{D\to C}_{xu}=w_{cx}
 \qquad
 \bigl(u\in C\cup D,\ x\notin C\cup D\bigr).
\end{equation*}
This is called \emph{symmetrizing} \(D\) to \(C\) in \(w\).
Let \(W^{D\to C}\), \(\Delta^{D\to C}\) and \(\Phi^{D\to C}\) be the corresponding sums after symmetrization.
Define \(\mathbf{w}^{C\to D}\) analogously. Both new weightings satisfy \eqref{dec:eq:2.2}: a positive-support
clique contains at most one vertex of \(C\cup D\), and replacing that vertex by a representative of the copied
class produces a positive-support clique of the same weight in \(\mathbf{w}\).
Moreover, \(C\cup D\) is one equivalence class in each new weighting, so the number of classes decreases.

For any function \(\phi:[0,1]\to\mathbb R\) with \(\phi(0)=0\), we have
\begin{equation*}
  \sum_{i<j}\phi(w_{ij}) = \frac{|C|}{|C|+|D|}\sum_{i<j}\phi(w^{D\to C}_{ij}) + \frac{|D|}{|C|+|D|}\sum_{i<j}\phi(w^{C\to D}_{ij}).
\end{equation*}
Taking \(\phi(t)=t\) and \(\phi(t)=t(1-t)\) shows that both \(W\) and \(\Phi\) satisfy this identity:
\begin{equation*}
 \Delta=\frac{|C|}{|C|+|D|}\Delta^{D\to C}+\frac{|D|}{|C|+|D|}\Delta^{C\to D},
 \qquad \Delta^{D\to C},\Delta^{C\to D}\ge0,
\end{equation*}
\begin{equation*}
 \Phi=\frac{|C|}{|C|+|D|}\Phi^{D\to C}+\frac{|D|}{|C|+|D|}\Phi^{C\to D}.
\end{equation*}
Since the square root is concave, suppose we have proved \eqref{dec:eq:assertion} for both \(\mathbf{w}^{D\to C}\) and \(\mathbf{w}^{C\to D}\), then
\begin{equation*}
\begin{aligned}
  \frac{\Phi}{n^2} &= \frac{|C|}{|C|+|D|}\frac{\Phi^{D\to C}}{n^2}+\frac{|D|}{|C|+|D|}\frac{\Phi^{C\to D}}{n^2}\\
 &\le \frac{|C|}{|C|+|D|}7\sqrt{\Delta^{D\to C}}+\frac{|D|}{|C|+|D|}7\sqrt{\Delta^{C\to D}}\\
 &\le 7\sqrt{\frac{|C|}{|C|+|D|}\Delta^{D\to C}+\frac{|D|}{|C|+|D|}\Delta^{C\to D}} = 7\sqrt{\Delta}.
\end{aligned}
\end{equation*}

By induction on the number of equivalence classes, it suffices to prove \eqref{dec:eq:assertion} for
every weighting that can be obtained from \(\mathbf{w}\) by repeated symmetrization. In particular,
we may assume that every pair of distinct equivalence classes has positive weight between them. 
If there is only one class, then all weights vanish and the assertion is immediate.
Hence assume that there are \(k\ge2\) classes \(C_1, \cdots, C_k\) with proportions \(p_1\ge p_2\ge\cdots\ge p_k>0\).
Without loss of generality, we may assume the vertex \(i\) is in \(C_i\) for \(i\in[k]\).
Then \(w_{ij}>0\) is the weight between classes \(i,j\), which satisfies \eqref{dec:eq:2.2} for every nonempty \(S\subseteq[k]\):
\begin{equation*}
 \sum_{\substack{i<j\\i,j\in S}}w_{ij}\le |S|-1
 \qquad(\varnothing\ne S\subseteq[k]).
\end{equation*}
Thus \(\overline{\mathbf{w}}=(w_{ij})_{i,j\in[k]}\) lies in Edmonds' forest polytope.
Therefore there is a set of forests on \([k]\) such that \(\overline{\mathbf{w}}\) is
a convex combination of their incidence vectors.

For any fixed forest \(F\) on \([k]\), root the component of \(F\) containing vertex \(1\) at \(1\), root every
other component arbitrarily, and denote the parent of a nonroot \(j\) by
\(\pi_F(j)\). We have
\begin{equation}\label{dec:eq:fixed forest}
\begin{aligned}
p_1(1-p_1) - \sum_{ij\in E(F)}p_ip_j &= p_1\sum_{j\ne1}p_j - \sum_{j\ {\rm nonroot}}p_{\pi_F(j)}p_j\\
&= \sum_{\substack{j\ne1\\j\ {\rm root}}}p_1p_j  +\sum_{\substack{j\ne1\\j\ {\rm nonroot}}}(p_1-p_{\pi_F(j)})p_j\\
&\ge0.
\end{aligned}
\end{equation}

Taking convex combination gives
\begin{equation}\label{dec:eq:conv comb}
\begin{aligned}
  p_1(1-p_1) - \sum_{\substack{i<j\\i,j\in [k]}}p_ip_jw_{ij} &= p_1(1-p_1) - \frac{1}{n^2}\sum_{\substack{i<j\\i,j\in [n]}}w_{ij}\\
  &= p_1(1-p_1) - \frac{1}{n^2}\left(\frac14-\Delta\right)n^2\\
  &= \Delta - \left(p_1 - \frac{1}{2}\right)^2 \ge 0,
\end{aligned}
\end{equation}
\begin{equation*}
\left|p_1-\frac12\right|\le\sqrt{\Delta}.
\end{equation*}

Now since \(w_{ij}(1-w_{ij})\le1-w_{ij}\) on edges incident with \(1\) and
\(w_{ij}(1-w_{ij})\le w_{ij}\) on all other edges, we have
\begin{equation}\label{dec:eq:Phi bound}
\begin{aligned}
  \frac{\Phi}{n^2} &= \sum_{i<j}p_ip_jw_{ij}(1-w_{ij})\\
  &\le \sum_{j>1}p_1p_j(1-w_{1j}) + \sum_{2\le i<j}p_ip_jw_{ij}\\
  &= p_1(1-p_1) - \sum_{i<j}p_ip_jw_{ij} + 2\sum_{2\le i<j}p_ip_jw_{ij}\\
  &\le \Delta + 2\sum_{2\le i<j}p_ip_jw_{ij}.
\end{aligned}
\end{equation}

If \(\Delta=0\), then \(p_1=1/2\). In particular, equality holds in \eqref{dec:eq:conv comb}.
Suppose \(p_2<p_1\), then \(p_3>0\).
Since \(w_{23}>0\), there is some forest \(F\) in the convex combination where
\(23\) is an edge.
Then \eqref{dec:eq:fixed forest} holds strictly for this \(F\).
Consequently, \eqref{dec:eq:conv comb} is also strict, a contradiction. Hence \(p_2=p_1=1/2\).
There is no other class, and \(\Delta=0\) forces \(w_{12}=1\). Thus \(\Phi=0\) and the assertion holds.

Assume henceforth that \(\Delta>0\).  If \(\sqrt{\Delta}\ge1/4\), then
\begin{equation*}
 \frac{\Phi}{n^2}=\sum_{i<j}p_ip_jw_{ij}(1-w_{ij})\le\frac14\sum_{i<j}p_ip_j\le\frac18\le7\sqrt{\Delta}.
\end{equation*}
It remains to consider \(0<\sqrt{\Delta}<1/4\). If \(p_1-p_2\ge \sqrt{\Delta}\),
for every forest \(F\) on \([k]\) and every edge of \(F\) not incident with \(1\), it is easy to see that
\begin{equation*}
 \sqrt{\Delta}p_{\pi_F(j)}p_j
 \le (p_1-p_{\pi_F(j)})p_j.
\end{equation*}
\begin{equation*}
\begin{aligned}
\sqrt{\Delta}\sum_{\substack{j\ne1\\j\ {\rm nonroot}\\\pi_F(j)\ne1}}p_{\pi_F(j)}p_j
 &\le \sum_{\substack{j\ne1\\j\ {\rm nonroot}}}(p_1-p_{\pi_F(j)})p_j\\
 &\le \sum_{\substack{j\ne1\\j\ {\rm root}}}p_1p_j  +\sum_{\substack{j\ne1\\j\ {\rm nonroot}}}(p_1-p_{\pi_F(j)})p_j
\end{aligned}
\end{equation*}
By \eqref{dec:eq:conv comb}, taking convex combination yields
\begin{equation*}\sqrt{\Delta}\sum_{2\le i<j}p_ip_jw_{ij}\le\Delta\end{equation*}
and hence
\begin{equation*}
 \frac{\Phi}{n^2}\le \Delta+2\sum_{2\le i<j}p_ip_jw_{ij}\le 3\sqrt{\Delta}.
\end{equation*}

If \(p_1-p_2<\sqrt{\Delta}\), then \(1-p_1-p_2\le3\sqrt{\Delta}\). Since \(w_{ij}\le1\) and \(p_2\le p_1\le1/2+\sqrt{\Delta}\),
\begin{equation*}
 \sum_{2\le i<j}p_ip_jw_{ij}\le p_2(1-p_1-p_2)+\frac{1}{2}(1-p_1-p_2)^2\le \frac32\sqrt{\Delta}+\frac{15}{2}\Delta.
\end{equation*}
It follows from \eqref{dec:eq:Phi bound} that
\begin{equation*}
 \frac{\Phi}{n^2}\le \Delta+2\sum_{2\le i<j}p_ip_jw_{ij}\le3\sqrt{\Delta}+16\Delta\le7\sqrt{\Delta}.
\end{equation*}
This proves the estimate in the remaining case.
\end{proof}

We can now prove stability for the weighted Mantel's theorem, which says that,
under the conditions of Lemma~\ref{dec:lem:weighted-mantel}, if the total weight \(W\) is close to optimal,
then the weighting \(\mathbf w\) is almost complete balanced bipartite.

\begin{lemma}
  \label{dec:lem:weighted-stability}
For every \(\theta>0\) there is \(\delta>0\) such that, for all
sufficiently large \(n\), let \(\mathbf{w}\) be an edge weighting
satisfying \emph{\eqref{dec:eq:2.2}} and
\begin{equation}\label{dec:eq:2.6}
  \sum_{i<j}w_{ij}\ge \left(\frac{1}{4}-\delta\right)n^2.
\end{equation}
Then there is a balanced bipartition \([n]=A\dotcupunion B\) for which
\begin{equation}\label{dec:eq:2.7}
  \sum_{a\in A,\ b\in B}(1-w_{ab})\le\theta n^2.
\end{equation}
\end{lemma}

\begin{proof}
Fix \(\theta>0\).  Choose \(0<\delta<1/16\) so small that, with \( \eta_\delta=\delta+21\sqrt{\delta}\), we have
\begin{equation*}
 2\eta_\delta+\sqrt{\eta_\delta}
 +\frac{21}{2}\sqrt{\delta}<\frac{\theta}{2}.
\end{equation*}
Take \(n\) sufficiently large that \(1/n<\theta/2\). Let \(W,\Delta,\Phi\) be as in
Lemma~\ref{dec:lem:weight-near-integrality}. Lemma~\ref{dec:lem:weighted-mantel}, \eqref{dec:eq:2.6},
and Lemma~\ref{dec:lem:weight-near-integrality} give
\begin{equation}\label{dec:eq:2.9}
 0\le\Delta\le\delta,\qquad
 \Phi\le7\sqrt{\Delta}\,n^2.
\end{equation}

Let \(H\) be the graph whose edges are the pairs \(ij\) with \(w_{ij}>2/3\).
Then \(H\) is triangle-free, since otherwise the total weight on a triangle would be greater than \(2\),
contrary to \eqref{dec:eq:2.2}. For every \(t\in[0,1]\),
\begin{equation*}
 \mathbf1_{\{t>2/3\}}\ge t-3t(1-t).
\end{equation*}
Consequently, if \(\eta=\Delta+21\sqrt{\Delta}\), then
\begin{equation}\label{dec:eq:2.10}
\begin{aligned}
 e(H)&=\sum_{i<j}\mathbf1_{\{w_{ij}>2/3\}}\\
 &\ge\sum_{i<j}w_{ij}-3\sum_{i<j}w_{ij}(1-w_{ij})\\
 &= W-3\Phi\\
 &\ge\left(\frac14-\eta\right)n^2.
\end{aligned}
\end{equation}

Choose a vertex \(v\) of maximum degree \(d\) in \(H\). Put
\(A_0=N_H(v)\) and \(B_0=V(H)\setminus A_0\).  The set \(A_0\) is
independent. We have
\begin{equation*}
\begin{aligned}
 \sum_{b\in B_0}d_H(b)
 &=e_H(A_0,B_0)+2e_H(B_0)\\
 &=e(H)+e_H(B_0)\\
 &\le d(n-d)\le n^2/4.
\end{aligned}
\end{equation*}
Thus
\begin{equation*}
 0\le e_H(B_0)\le d(n-d)-e(H)\le\eta n^2,\qquad
 d(n-d)-e_H(A_0,B_0)\le2\eta n^2.
\end{equation*}
Also \(e(H)\le d(n-d) = \frac{n^2}{4}-\left|d-\frac n2\right|^2\), so \eqref{dec:eq:2.10} gives
\begin{equation*}
 \left|d-\frac n2\right|\le\sqrt{\eta}\,n.
\end{equation*}
Move at most \(\sqrt{\eta}\,n+1\) vertices between \(A_0\) and \(B_0\)
to obtain a balanced bipartition \(A\dotcupunion B\). Recall that
\(m_H(A,B)\) denotes the number of missing \(H\)-edges across this bipartition.
Moving one vertex creates at most \(n\) new missing crossing edges. Thus
\begin{equation}\label{dec:eq:2.12}
 m_H(A,B)\le(2\eta+\sqrt{\eta})n^2+n.
\end{equation}

For every edge \(ab\in E(H)\), the inequality \(w_{ab}>2/3\) gives
\begin{equation*}
 1-w_{ab}\le\frac32w_{ab}(1-w_{ab}).
\end{equation*}
The pairs missing from \(H\) contribute at most one each, so \eqref{dec:eq:2.9} and
\eqref{dec:eq:2.12} yield
\begin{equation*}
\begin{aligned}
 \sum_{a\in A,\ b\in B}(1-w_{ab})
 &\le m_H(A,B)+\frac32\Phi\\
 &\le
 \left(2\eta+\sqrt{\eta}
       +\frac{21}{2}\sqrt{\Delta}\right)n^2+n\\
 &\le
 \left(2\eta_\delta+\sqrt{\eta_\delta}
       +\frac{21}{2}\sqrt{\delta}+\frac1n\right)n^2
 <\theta n^2.
\end{aligned}
\end{equation*}
This is \eqref{dec:eq:2.7}.
\end{proof}

Recall that the linear-programming dual of the fractional edge decomposition problem is
\begin{equation}\label{dec:eq:2.4}
  \CDstar_2(G,(i-1))=\max\left\{\sum_{e\in E(G)}y_e:
  \sum_{e\in E(K)}y_e\le |K|-1\text{ for every clique }K, y_e\in \R\right\}.
\end{equation}
Let \(\mathbf{y}\) be an optimal solution of \emph{\eqref{dec:eq:2.4}}.
Extend it symmetrically to \([n]\times[n]\) and let \(\mathbf{y}^+,\mathbf{y}^-\) be
its positive and negative parts respectively:
\begin{equation*}
  y_e^+=\max\{y_e,0\},\quad y_e^-=\max\{-y_e,0\}\quad \text{for\ } e\in E(G);
  \qquad Y^+=\sum_e y_e^+,\quad Y^-=\sum_e y_e^-.
\end{equation*}
Then \(\CDstar_2(G,(i-1))=\sum_e y_e=Y^+-Y^-\). Put \(y_e=0\) on nonedges.
The constraint on single edge \(e\) gives \(y_e^+\le1\).
By definition, \(\mathbf{y}^+\) satisfies \emph{\eqref{dec:eq:2.2}},
so Lemma~\ref{dec:lem:weighted-mantel}--\ref{dec:lem:weighted-stability} apply.
We are now ready to prove Theorem~\ref{dec:thm:fractional-stability}.

\begin{proof}[Proof of Theorem~\ref{dec:thm:fractional-stability}]
Choose \(\theta>0\) so small that
\(\theta+\frac12\sqrt\theta<\eta\), and let \(\delta_0>0\) be supplied by Lemma~\ref{dec:lem:weighted-stability}.
Take \(\delta<\delta_0/2\) small enough so that
\begin{equation}\label{dec:eq:2.14}
  \theta+\frac12\sqrt{\theta+\delta}<\eta.
\end{equation}
Enlarge \(n_0\) beyond the threshold in Lemma~\ref{dec:lem:weighted-stability} and
so that every rounding estimate below holds.

Let \(G\) satisfy the hypothesis of Theorem~\ref{dec:thm:fractional-stability},
and let \(\mathbf{y},\mathbf{y}^+,\mathbf{y}^-\) be as above. Since
\begin{equation*}
  \sum_e y_e^+=\CDstar_2(G,(i-1))+Y^-
  \ge \lfloor \frac{n^2}{4}\rfloor-\delta n^2>(\frac14-\delta_0)n^2.
\end{equation*}
Lemma~\ref{dec:lem:weighted-stability} gives a balanced bipartition \(A\dotcupunion B\) with
\begin{equation*}
Z_{\mathbf{y}^+}:=\sum_{a\in A,\ b\in B}(1-y^+_{ab})\le\theta n^2.
\end{equation*}
Every missing crossing pair in \(G\) has weight zero in \(\mathbf y\), so
\begin{equation}\label{dec:eq: miss cross}
 m_G(A,B)\le Z_{\mathbf{y}^+}\le\theta n^2.
\end{equation}

For \(e=aa'\in E(G[A])\) and \(f=bb'\in E(G[B])\), let \(R(e,f)\) be the four crossing
pairs between the endpoints.  If one of these pairs is a nonedge, then
\(\sum_{g\in R(e,f)}(1-y^+_g)\ge1\).  If all four are edges, the dual constraint on the \(K_4\) gives
\begin{equation}\label{dec:eq:2.15}
  1\le
  \sum_{g\in R(e,f)}(1-y^+_g)
  +y_e^-+y^-_f+\sum_{g\in R(e,f)}y^-_g.
\end{equation}
Sum \eqref{dec:eq:2.15} over \(E(G[A])\times E(G[B])\), adding the nonnegative \(\mathbf{y}^-\)-terms also when
\(R(e,f)\) contains a nonedge. Every multiplicity is at most \(n^2/4\), so
\begin{equation*}
  e(G[A])e(G[B])\le \frac{n^2}{4}(Z_{\mathbf{y}^+}+Y^-).
\end{equation*}
Interchange the sides so that \(e(G[A])\le e(G[B])\) if necessary, then
\begin{equation}\label{dec:eq:2.17}
  e(G[A])\le \frac n2\sqrt{Z_{\mathbf{y}^+}+Y^-}.
\end{equation}

Since
\begin{equation*}
Y^-=Y^+-\CDstar_2(G,(i-1))\le \lfloor\frac{n^2}{4}\rfloor-\CDstar_2(G,(i-1))\le  \delta n^2,
\end{equation*}
combining this with \eqref{dec:eq: miss cross} and \eqref{dec:eq:2.17} gives
\begin{equation*}
  m_G(A,B)+e(G[A])
  \le \theta n^2+\frac12\sqrt{\theta+\delta}n^2<\eta n^2
\end{equation*}
by \eqref{dec:eq:2.14}.  This is \eqref{dec:eq:2.1}.
\end{proof}

\section{Exact decomposition near a balanced complete bipartition}\label{dec:sec:exact-near-cut}

We now construct a clique decomposition step by step based on the balanced bipartition \(V(G)=A\dotcupunion B\)
supplied by Theorem~\ref{dec:thm:fractional-stability}. We first exclude a few vertices with high internal degrees
or high missing crossing degrees, which we call exceptional. This will be made formal in
Theorem~\ref{dec:thm:stable-decomposition}. Then we work on the remaining subsets \(C\subseteq A\) and \(D\subseteq B\)
with equal size. Regard \(C\times D\) as the edge set of the complete bipartite graph \(K_{C,D}\).
For \(W\subseteq C\times D\), write \(\Delta_D(W):=\max_{d\in D}d_W(d)\). A \(C\)--\(D\) pair is forbidden
when it is a nonedge in \(G\)(recorded in \(M\)) or has already been used by a previously selected clique(recorded in \(U\)).

As suggested by Theorem~\ref{dec:thm:fractional-stability}, the induced graph on \(C\) is sparse.
So we just greedily select one common neighbor in \(D\) for each edge in \(G[C]\) to form a triangle.
The construction also controls the number of selected crossing edges incident with each vertex of \(D\).

\begin{lemma}
\label{dec:lem:triangle-extension}
Let \(F\) be a graph on \(C\), and let \(M,U\subseteq C\times D\), where
the pairs in \(M\cup U\) are forbidden. Suppose that \(a,b\) are nonnegative integers such that
\begin{equation*}
 d_F(c)+d_{M\cup U}(c)\le a\qquad \text{for\ }c\in C
\end{equation*}
and
\begin{equation}\label{dec:eq:3.3}
 |D|>2a+\frac{|U|+2e(F)}{b+1}.
\end{equation}
Then \(E(F)\) can be covered by edge-disjoint triangles \(cc'd\), with
\(cc'\in E(F)\) and \(d\in D\), such that the set
\(W\subseteq C\times D\) of crossing edges in these triangles is disjoint
from \(M\cup U\) and satisfies
\begin{equation}\label{dec:eq:3.4}
 d_W(c)=d_F(c)\quad(c\in C),\qquad
 \Delta_D(U\cup W)\le\max\{\Delta_D(U),b+2\}.
\end{equation}
In particular, the triangles have total cost
\begin{equation*}
  \cost(\mathcal T)=|W|=2e(F).
\end{equation*}
\end{lemma}

\begin{proof}
Process the edges of \(F\) in any order.  Suppose that \(W'\) is the set
of crossing edges already used and that \(cc'\) is the next edge.  Vertices
\(d\in D\) for which \(cd\) or \(c'd\) belongs to
\(M\cup U\cup W'\) number at most
\begin{equation*}
 \bigl(d_{M\cup U}(c)+d_{W'}(c)\bigr)
 +\bigl(d_{M\cup U}(c')+d_{W'}(c')\bigr)\le2a,
\end{equation*}
because \(d_{W'}(x)\le d_F(x)\).  Moreover,
\begin{equation*}
 \bigl|\{d\in D:d_{U\cup W'}(d)\ge b+1\}\bigr|
 \le\frac{|U|+|W'|}{b+1}
 \le\frac{|U|+2e(F)}{b+1}.
\end{equation*}
Thus \eqref{dec:eq:3.3} leaves a vertex \(d\) for which both crossing edges are available
and \(d_{U\cup W'}(d)\le b\).  Add the triangle \(cc'd\).  Its two crossing
edges are new, and the degree of \(d\) becomes at most \(b+2\).  Induction gives an
edge-disjoint packing with \eqref{dec:eq:3.4}.  Each incidence of \(c\) with an edge
of \(F\) uses exactly one crossing edge at \(c\), so \(d_W(c)=d_F(c)\).
\end{proof}

If \(D\) is also sparse, \(G[D]\) could be decomposed with the same method. In case it is dense,
we will first decompose \(G[D]\) into triangles and edges, resulting in a hypergraph \(\mathcal H\).
Then we use vertices in \(C\) as colors to give a proper edge-coloring of \(\mathcal H\).
A hyperedge \(e\) could only use a \(c\in C\) if all of its vertices are adjacent to \(c\) so that
they could form a clique together. For notation, the rank of a hypergraph is the maximum size of any
of its hyperedges, and its maximum codegree is
\begin{equation*}
 \Delta_2(\mathcal H)
 :=\max_{u\ne v}\bigl|\{e\in E(\mathcal H):u,v\in e\}\bigr|.
\end{equation*}
We use Kahn's list-edge-coloring theorem~\cite{Kahn}. 
\begin{theorem}[Kahn~\cite{Kahn}]\label{dec:thm:Kahn}
For every fixed integer \(r\ge2\) and every \(\eta>0\), there are \(\delta>0\) and \(D_0\) such that
the following holds.  If a hypergraph \(\mathcal H\) has rank at most \(r\), maximum degree \(D\ge D_0\), and
\begin{equation*}
 \Delta_2(\mathcal H)\le\delta D,
\end{equation*}
then every assignment of a list \(L(e)\) of at least \((1+\eta)D\) colors to each \(e\in E(\mathcal H)\) admits
a proper edge-coloring \(\phi\) with \(\phi(e)\in L(e)\) for every \(e\).
\end{theorem}

The next lemma covers edges in \(G[D]\) when it is dense.

\begin{lemma}
  \label{dec:lem:dense-completion}
For every \(0<\beta<1/2\),
there is \(\e_0>0\) such that, for every fixed
\(0<\e\le\e_0\) and all sufficiently large
integers \(q\), the following holds.
Let \(C,D\) be disjoint \(q\)-sets, let \(H\) be a graph on \(D\) with
\begin{equation}\label{dec:eq:3.5}
  e(H)\ge\beta q^2,
\end{equation}
and let \(M,U\subseteq C\times D\) be sets of forbidden crossing pairs satisfying
\begin{equation*}
  \Delta_D(M\cup U)\le3\e q.
\end{equation*}
Then \(E(H)\) has an exact decomposition into triangles with one vertex
in \(C\) and \(K_4\)'s with one vertex in \(C\), using pairwise distinct
crossing edges outside \(M\cup U\), such that the total cost of these cliques equals
the number of crossing edges they use.
\end{lemma}

\begin{proof}
Choose \(\e_0>0\) so that
\begin{equation}\label{dec:eq:3.7}
 60\e_0<\beta.
\end{equation}

We first construct a family \(\mathcal P\) of edge-disjoint triangles in
\(H\).  For \(v\in D\), let \(t_{\mathcal P}(v)\) be the number of members
of \(\mathcal P\) containing \(v\).  Starting with
\(\mathcal P=\varnothing\), for every vertex \(v\) with \(d_H(v)>(1-10\e) q\), as long as
\begin{equation*}
t_{\mathcal P}(v)<\lceil10\e q\rceil-1,
\end{equation*}
choose a new triangle through \(v\) and charge
that triangle to \(v\). If no such vertex exists initially, then \(\mathcal P\) remains empty.
Each vertex receives at most \(\lceil10\e q\rceil-1\) charges, so
throughout the process
\begin{equation*}
 |\mathcal P|\le\sum_{v\in D}(\lceil10\e q\rceil-1)<10\e q^2.
\end{equation*}

It remains to verify that the required new triangle always exists.
By \eqref{dec:eq:3.5} and \eqref{dec:eq:3.7},
\begin{equation}\label{dec:eq:3.11}
 e\bigl(H[N_H(v)]\bigr)
 \ge e(H)-q|D\setminus N_H(v)|
 >\left(\beta-10\e\right)q^2
 >50\e q^2.
\end{equation}
A candidate edge \(xy\) in \(H[N_H(v)]\) fails to give an edge-disjoint
new triangle \(vxy\) only if \(xy\) is already used in \(\mathcal P\) or one of \(x,y\) lies in
\(N_{\mathcal P}(v)\).  Thus at most
\begin{equation*}
 3|\mathcal P|+q d_{\mathcal P}(v)<30\e q^2+2qt_{\mathcal P}(v)<50\e q^2
\end{equation*}
candidate edges fail. \eqref{dec:eq:3.11} therefore supplies a suitable
triangle.  The process terminates with
\begin{equation}\label{dec:eq:3.12}
 t_{\mathcal P}(v)\ge \lceil10\e q\rceil-1\qquad \text{for\ }v\in D,\ d_H(v)>(1-10\e) q.
\end{equation}

Define a hypergraph \(\mathcal H\) on vertex set \(D\).  Its \(3\)-edges
are the vertex sets of the triangles in \(\mathcal P\), and its \(2\)-edges
are the edges of \(H\) that belong to no member of \(\mathcal P\).  The
triangles in \(\mathcal P\) are edge-disjoint, so any two hyperedges of
\(\mathcal H\) meet in at most one vertex; that is, \(\mathcal H\) is
linear.  For every \(v\in D\), \eqref{dec:eq:3.12} gives
\begin{equation*}
 d_{\mathcal H}(v)
 =t_{\mathcal P}(v)+d_H(v)-2t_{\mathcal P}(v)
 =d_H(v)-t_{\mathcal P}(v)
 \le(1-10\e)q.
\end{equation*}
Moreover,
\begin{equation*}
 \sum_{v\in D}d_{\mathcal H}(v)
 =2\bigl(e(H)-3|\mathcal P|\bigr)+3|\mathcal P|
 =2e(H)-3|\mathcal P|
 >\left(2\beta-30\e\right)q^2
 >90\e q^2.
\end{equation*}
Therefore, 
\begin{equation}\label{dec:eq:3.15}
 90\e q<\Delta({\mathcal H})\le(1-10\e)q,
 \qquad \Delta_2(\mathcal H)\le1.
\end{equation}

For each hyperedge \(e\in E(\mathcal H)\), assign the list
\begin{equation*}
 L(e)=\{c\in C:(c,x)\notin M\cup U\text{ for every }x\in e\}.
\end{equation*}
Since \(|e|\le3\) and every vertex of \(D\) has degree at most
\(3\e q\) in \(M\cup U\),
\begin{equation*}
 |L(e)|\ge(1-9\e)q.
\end{equation*}
Apply Theorem~\ref{dec:thm:Kahn} with rank at most \(3\) and
\(\eta=\e\).  For fixed \(\e>0\), \eqref{dec:eq:3.15} ensures, once \(q\) is
sufficiently large, both \(\Delta({\mathcal H})\ge D_0\) and
\(1\le\delta \Delta({\mathcal H})\).  Finally, for sufficiently large \(q\), \eqref{dec:eq:3.15} gives
\begin{equation*}
 (1+\e)\Delta({\mathcal H})
 \le(1+\e)(1-10\e)q
 =(1-9\e-10\e^2)q
 <(1-9\e)q.
\end{equation*}
Thus there is a proper edge-coloring \(\phi:E(\mathcal H)\to C\) with \(\phi(e)\in L(e)\) for every hyperedge \(e\).

If a \(2\)-edge \(e=\{u,v\}\) has color \(c\), take the triangle on
\(\{c,u,v\}\).  If a \(3\)-edge \(e=\{u,v,w\}\) has color \(c\),
take the \(K_4\) on \(\{c,u,v,w\}\).  Membership of \(c\) in \(L(e)\) ensures that every
required crossing edge is present and unused. The resulting cliques cover every edge of
\(H\) exactly once. If two resulting cliques used
the same crossing edge \(cx\), their hyperedges would both contain \(x\) and
would both have color \(c\), contrary to the properness of \(\phi\). So the asserted total cost identity follows.

\end{proof}

We then turn to the exceptional vertices, that is, the remaining vertices outside \(C\cup D\).
We will reassign these vertices to the two sides depending on their degrees in each side. We prefer each vertex to have smaller degree in its own side.
Then we cover all edges incident to the exceptional vertices.

We will repeatedly use the following robust form of Hall's theorem for a nearly complete
bipartite graph.  It bounds the matching deficiency using only the maximum degree of the set of forbidden pairs.

\begin{lemma}
  \label{dec:lem:hall}
Let \(A,B\) be finite sets with \(|A|\le |B|\), and let
\(W\subseteq A\times B\) have maximum degree at most \(z\).
Then \(K_{A,B}-W\) has a matching leaving at most
\begin{equation*}
  \max\{0,2z-|B|\}
\end{equation*}
vertices of \(A\) unmatched.  In particular, it saturates \(A\) whenever
\(|B|\ge2z\).
\end{lemma}

\begin{proof}
For nonempty \(X\subseteq A\) with \(N(X)\ne B\), choose
\(x\in X\) and \(y\in B\setminus N(X)\).  Since
\(X\times\{y\},\{x\}\times(B\setminus N(X))\subseteq W\), the degree
bound gives \(|X|\le z\) and \(|B\setminus N(X)|\le z\).  Thus the Hall
deficiency is at most \(2z-|B|\); the empty case and the case \(N(X)=B\)
are immediate.
\end{proof}

We now give the exact construction for the exceptional vertices.

\begin{lemma}
  \label{dec:lem:exceptional-extension}
Let \(G\) be a graph. Let \(C,D\) be disjoint \(q\)-sets of \(V(G)\) and let \(R:=V(G)\setminus(C\dotcupunion D)\) be a disjoint set of \(N\)
exceptional vertices.
Let \(M\subseteq C\times D\) be
the set of missing \(C\)--\(D\) pairs, and let
\(U\subseteq C\times D\) be the set of crossing edges used by a
previously selected edge-disjoint family of cliques of total cost
\(|U|\). Assume the previously selected family contains no \(G[R]\) or \(G[R, C\cup D]\) edge.  Thus every member of \(U\) is present in \(G\).
Suppose
\begin{equation*}
  \Delta_{C\cup D}(M\cup U)\le z.
\end{equation*}
If
\begin{equation*}
  q>2z+8N,
\end{equation*}
then there is a bipartition \(R=S\dotcupunion T\) and an exact clique
decomposition of \(E(G[R])\), all edges between \(R\) and \(C\cup D\), and
all present \(C\)--\(D\) edges outside \(U\).  For the final bipartition
\begin{equation*}
  P=C\dotcupunion S,\qquad Q=D\dotcupunion T,
\end{equation*}
adjoining the previously selected cliques gives total cost at most
\begin{equation*}
  q^2+qN+e_{G[R]}(S,T).
\end{equation*}
\end{lemma}

\begin{proof}
If \(N=0\), retain every present \(C\)--\(D\) edge outside \(U\) as a
\(2\)-clique, and set \(R=R_1=R_2=S=T=\varnothing\), \(s=r=m=0\).
After adjoining the previously selected cliques, the cost is \(q^2-|M|\), so
the exact identity \eqref{dec:eq:3.34} below holds.  Thus assume \(N\ge1\).

For \(v\in R\), put
\begin{equation*}
  l_v=\min\{d(v,C),d(v,D)\},\qquad
  h_v=\max\{d(v,C),d(v,D)\}.
\end{equation*}
Let \(L_v\) and \(H_v\) be the members of \(\{C,D\}\) with, respectively, \(l_v\) and \(h_v\) neighbors of \(v\).
If \(l_v=h_v\), simply choose \(L_v=C\) and \(H_v=D\). Initially assign \(v\) to the side containing \(L_v\). Let
\begin{equation}\label{dec:eq:def R1R2R3}
  \begin{aligned}
    R_1 &:= \{v\in R:q-h_v\ge N\},\\
    R_2 &:= \{v\in R:q-h_v<N, h_v-l_v> N\},\\
    R_3 &:= \{v\in R:q-h_v<N, h_v-l_v\le N\}.
  \end{aligned}
\end{equation}
Keep every member of \(R_1\) and \(R_2\) on its initial side. For \(v\in R_3\), let
\begin{equation*}
  \begin{aligned}
    I_v &:= \{u\in N_{R_2}(v):L_u=L_v\},\\
    O_v &:= \{u\in N_{R_2}(v):L_u=H_v\}.
  \end{aligned}
\end{equation*}
Keep \(v\) on its initial side if
\begin{equation*}
  |I_v|-|O_v|\le h_v-l_v,
\end{equation*}
and move it to the other side otherwise.

Now we have partitioned \(R=S\dotcupunion T\) into two sides. Let \(P=C\dotcupunion S\) and \(Q=D\dotcupunion T\).
Let \(P\dotcupunion Q\) be the final bipartition. The edges of \(G\) except those already used in \(U\) are decomposed
by the following steps carried out in order. Throughout the construction, we require that each \(c\in C\)(and each \(d\in D\))
uses at most \(N\) \(C\)--\(D\) edges in each step except for step~\ref{dec:step:h}.
We call this the \emph{maximum degree constraint}.

\begin{enumerate}[label=\textup{(\alph*)}]

\item Retain every edge of \(G[R]\) incident to \(R_1\) as a \(2\)-clique and assign it to one endpoint in \(R_1\),
choosing one endpoint arbitrarily if both lie in \(R_1\). Each vertex of \(R_1\) receives at most \(N-1\) assignments.\label{dec:step:a}

\item For every \(uv\in E(R_2)\), if \(L_u=L_v\), use a triangle through a common neighbor of \(u,v\) in \(H_u=H_v\).
There are at least \(q-4N\) unused choices for that common neighbor. Retain \(uv\) as a crossing \(2\)-clique otherwise.\label{dec:step:b}

\item For every \(uv\in E(R_3)\), choose a present unused \(C\)--\(D\) edge \(cd\) such that both \(u\) and \(v\) are
adjacent to both \(c\) and \(d\). Use the \(K_4\) on \(\{u,v,c,d\}\). This is possible because \(u\) and \(v\) have at
least \(q-6N\) unused common neighbors in each of \(C\) and \(D\), among which at most \(\binom{N}{2}/N<N\) are exhausted
by the maximum degree constraint. At each of the remaining \(q-7N\) choices for \(c\in C\cap N(u)\cap N(v)\),
there are at least \((q-7N)-N-z\) choices for \(d\in D\cap N(u)\cap N(v)\) such that \(cd\) is present and unused.\label{dec:step:c}

\item For every \(v\in R_3\), suppose first that \(v\) is kept on its initial side.
For each \(u\in I_v\), \(H_u=H_v\). Choose \(x\in H_u\cap N(u)\cap N(v)\) and use the triangle \(vux\).
And for each \(w\in O_v\), \(H_w=L_v\). Choose \(y\in H_w\cap N(w)\cap N(v)\) and use the triangle \(vwy\).
Each \(u\)(or \(w\)) and \(v\) have at least \(q-3N\) common neighbors in \(H_u\)(or \(H_w\)),
among which at most \(2N\) are previously used. Thus there are at least \(q-5N\) choices for the third vertex of each triangle.

Suppose now that \(v\) is moved to the other side, then  \(|I_v|>|O_v|+h_v-l_v\). Thus \(O_v\) is on the same side
as \(v\) and \(I_v\) is on the opposite side to \(v\) now. Choose an arbitrary \(I'_v\subseteq I_v\) with \(|I'_v|=|O_v|+h_v-l_v\).
For each \(u\in I'_v\), \(H_u=H_v\). Choose \(x\in H_u\cap N(u)\cap N(v)\) and use the triangle \(vux\).
And for each \(w\in O_v\), \(H_w=L_v\). Choose \(y\in H_w\cap N(w)\cap N(v)\) and use the triangle \(vwy\).
The construction for each triangle is well-defined for the same reason as above. Retain all remaining crossing \(v\)--\(I_v\) edges as \(2\)-cliques.\label{dec:step:d}

\item For every \(v\in R_3\), step~\ref{dec:step:c} uses \(d_{R_3}(v)\) \(v\)--\(L_v\) and \(v\)--\(H_v\) edges each.
If \(v\) is kept on its initial side, step~\ref{dec:step:d} uses \(|O_v|\) \(v\)--\(L_v\) edges and \(|I_v|\) \(v\)--\(H_v\) edges.
If \(v\) is moved to the other side, step~\ref{dec:step:d} uses \(|O_v|\) \(v\)--\(L_v\) edges and \(|O_v|+h_v-l_v\) \(v\)--\(H_v\) edges.
Let the remaining unused neighbors of \(v\) in \(H_v\) and \(L_v\) be \(H'_v\) and \(L'_v\) respectively.
Then in both cases \(v\) has at least as many crossing neighbors as internal neighbors in \(C\cup D\) now.
We claim that Lemma~\ref{dec:lem:hall} gives a matching in \(G[L'_v,H'_v]\) saturating the smaller side.
The graph consisting of forbidden \(L'_v\)--\(H'_v\) edges has maximum degree at most \(z+2N\):
the forbidden edges from \(M\cup U\) contribute at most \(z\), the \(K_4\)'s constructed in step~\ref{dec:step:c} contribute
at most \(N\), and each processed vertex of \(R_3\) in step~\ref{dec:step:e} before \(v\) contributes at most one.
And previous analysis shows that the larger side contains at least \(q-4N>2(z+2N)\) vertices.
For every matched pair \(u\in L'_v,w\in H'_v\), use the triangle \(vuw\).
Retain all remaining edges incident to \(v\) that cross the final bipartition as \(2\)-cliques.\label{dec:step:e}

\item For every \(v\in R_2\), steps~\ref{dec:step:b} and \ref{dec:step:d} use at most \(N\) \(v\)--\(H_v\) edges in total.
Also let the remaining unused neighbors of \(v\) in \(H_v\) and \(L_v\) be \(H'_v\) and \(L'_v\) respectively.
Then \(v\) has at least
\begin{equation*}
h_v-N>q-2N
\end{equation*}
neighbors in \(H'_v\) now, which is at least as large as \(|L'_v|\). 
And the graph consisting of forbidden \(L'_v\)--\(H'_v\) edges has maximum degree at most \(z+2N\):
the forbidden edges from \(M\cup U\) contribute at most \(z\), the \(K_4\)'s constructed in step~\ref{dec:step:c} contribute
at most \(N\), and each processed vertex of \(R_3\) and \(R_2\) in steps~\ref{dec:step:e} and \ref{dec:step:f}
before \(v\) contributes at most one. So Lemma~\ref{dec:lem:hall} again gives a matching in \(G[L'_v,H'_v]\) saturating \(L'_v\).
For every matched pair \(u\in L'_v,w\in H'_v\), use the triangle \(vuw\).
Retain all remaining edges incident to \(v\) that cross the final bipartition as \(2\)-cliques.\label{dec:step:f}

\item For every \(v\in R_1\), the graph consisting of forbidden \(C\)--\(D\) edges has maximum degree at most \(z+2N\):
the forbidden edges from \(M\cup U\) contribute at most \(z\), the \(K_4\)'s constructed in step~\ref{dec:step:c} contribute
at most \(N\), and each previously processed vertex of \(R\) in steps~\ref{dec:step:e}--\ref{dec:step:g} before \(v\) contributes
at most one. Also let the remaining unused neighbors of \(v\) in \(H_v\) and \(L_v\) be \(H'_v\) and \(L'_v\) respectively.
Lemma~\ref{dec:lem:hall} gives a matching in \(G[L'_v,H'_v]\) with deficiency at most \(\max\{0,2z+4N-h_v\}\).
Use a triangle \(vuw\) for every edge \(uw\) of this matching.  Retain the unmatched edges from \(v\) to \(L_v\) and
every unmatched crossing edge from \(v\) as \(2\)-cliques.\label{dec:step:g}

\item Retain all remaining \(C\)--\(D\) edges as \(2\)-cliques.\label{dec:step:h}
\end{enumerate}

The construction covers each edge of \(G[R]\), each edge between \(R\) and \(C\cup D\),
and each present \(C\)--\(D\) edge outside \(U\) exactly once. Relative to the final bipartition \((P, Q)\),
every triangle we used has zero cost saving. A \(K_4\) constructed in step~\ref{dec:step:c} has saving one
exactly when its two exceptional vertices lie on opposite sides. Let \(s\) count such \(K_4\)'s.
A \(2\)-clique has one additional cost when both of its vertices lie on the same side, which,
according to our construction, is exactly when it is an internal \(2\)-clique constructed in
step~\ref{dec:step:a} or an unmatched internal \(2\)-clique constructed in step~\ref{dec:step:g}.
Let \(r\) count such internal \(2\)-cliques. Let \(m\) count the missing pairs between \(R\)
and \(C\cup D\) crossing the final bipartition. Every fixed \(v\in R_1\) contributes at most \(N-1\) internal
\(2\)-cliques in step~\ref{dec:step:a} and at most \(\max\{0,2z+4N-h_v\}\) unmatched internal \(2\)-cliques
in step~\ref{dec:step:g}. However, it also contributes \(q-h_v\) missing edges between \(R\) and \(C\cup D\). And we have
\begin{equation}\label{dec:eq:add vs save}
  N-1+\max\{0,2z+4N-h_v\}< q-h_v.
\end{equation}
Summing \eqref{dec:eq:add vs save} over \(R_1\) gives \(r\le m\) and equality holds only when \(R_1\) is empty.
Let \(\mathcal D\) be the union of the constructed decomposition and the previously selected decomposition of \(U\) which
has cost \(|U|\). Then \eqref{dec:eq:1.2} gives
\begin{equation}\label{dec:eq:3.34}
  \begin{aligned}
     \cost(\mathcal D)&=(q^2-|M|)+(qN-m)+e_{G[R]}(S,T)+r-s\\
     &\le q^2+qN+e_{G[R]}(S,T).
  \end{aligned}
\end{equation}
\end{proof}

The next theorem combines the preceding lemmas to obtain an exact clique decomposition given
a balanced bipartition and records the resulting cost through a nonnegative slack.

\begin{theorem}
\label{dec:thm:stable-decomposition}
There are absolute constants \(\xi>0\) and \(k^\ast\) such that the
following holds.  Let \(\sigma\in\{0,1\}\), \(k\ge k^\ast\), \(n=2k+\sigma\), and
\begin{equation*}
  V(G)=A\dotcupunion B,\qquad
  \{|A|,|B|\}=\{k,k+\sigma\}.
\end{equation*}
Let \(M\) be the set of missing \(A\)--\(B\) pairs. If
\begin{equation*}
  e(G[A])+|M|\le\xi k^2,
\end{equation*}
then there is an exact decomposition \(\mathcal D\) and a nonnegative integer \(\Omega\) such that
\begin{equation}\label{dec:eq:3.36}
 \cost(\mathcal D)=\lfloor \frac{n^2}{4}\rfloor-\Omega.
\end{equation}
The construction also produces a final bipartition \(V(G)=P\dotcupunion Q\).
Moreover, if \(\Omega=0\), then \((P,Q)\) are balanced and
\begin{equation*}
 G[P,Q]=K_{|P|,|Q|}.
\end{equation*}
\end{theorem}

\begin{proof}
\setcounter{proofcase}{0}
Put \(n=|V(G)|=2k+\sigma\).
Let \(\e_0\) be supplied by Lemma~\ref{dec:lem:dense-completion} with \(\beta=10^{-5}\).
Choose \(0<\e<\min\left\{\e_0, 10^{-3}\right\}\). Then choose
\begin{equation}\label{dec:eq:3.38}
  0<\xi\le\frac{\e^3}{4},
  \qquad \xi+2\beta\le10^{-4}.
\end{equation}
All inequalities below involve only these fixed constants, so one absolute threshold \(k^\ast\) will suffice.
\proofcase{dec:case:both-sparse}
  {Both \(G[A]\) and \(G[B]\) are sparse}
Suppose \(e(G[B])\le2\beta k^2\).  Then
\begin{equation*}
  e(G[A])+e(G[B])+|M|\le10^{-4}k^2.
\end{equation*}
Let 
\begin{equation*}
  \begin{aligned}
    X_0&=\{v\in A:d_{G[A]}(v)+d_M(v)\ge \lceil k/100\rceil\},\\
    Y_0&=\{v\in B:d_{G[B]}(v)+d_M(v)\ge \lceil k/100\rceil\}.
  \end{aligned}
\end{equation*}
Then
\begin{equation*}
 |X_0|+|Y_0|\le\frac{2\cdot10^{-4}k^2}{\lceil k/100\rceil}\le\frac k{50}.
\end{equation*}
Set \(l=\max\{|X_0|,|Y_0|+|A|-|B|\}\), enlarge \(X_0,Y_0\) to sets
\(X\subseteq A\), \(Y\subseteq B\) of orders \(l,l-|A|+|B|\), and put
\(C=A\setminus X\), \(D=B\setminus Y\).  Then, with \(R=X\cup Y\),
\begin{equation*}
 |C|=|D|=q=|A|-l\ge\frac{49k}{50},\qquad
 N:=|R|=2l-|A|+|B|.
\end{equation*}
Every vertex of \(C\cup D\) has internal degree plus missing crossing degree at most \(\lceil k/100\rceil-1\).

Apply Lemma~\ref{dec:lem:triangle-extension} to \(G[C]\), with \(M[C,D]\) forbidden,
no previously used crossing edges, and both parameters \(\lceil k/100\rceil-1\).
Its numerical condition holds for large \(k\), since
\begin{equation*}
 2(\lceil k/100\rceil-1)+\frac{2e(G[C])}{\lceil k/100\rceil}
 <\frac k{50}+\frac k{50}<q
\end{equation*}
for sufficiently large \(k\). Let \(\mathcal T_C\) be the resulting triangle family and \(U_C\) its set
of used crossing edges. The lemma gives
\begin{equation}\label{dec:eq:3.41}
 d_{U_C}(c)\le \lceil k/100\rceil-1\quad \text{for\ }c\in C,\qquad
 d_{U_C}(d)\le \lceil k/100\rceil+1\quad \text{for\ }d\in D.
\end{equation}

Apply the same lemma with the roles of \(C,D\) reversed, to \(G[D]\),
with \(M[C,D]\) missing, \(U_C\) as the previously used crossing edges, and parameters \(2\lceil k/100\rceil\) and
\(\lceil k/100\rceil-1\).  Indeed, \eqref{dec:eq:3.41} gives the required local bound \(2\lceil k/100\rceil\), and
\begin{equation*}
 2\cdot2\lceil k/100\rceil+\frac{|U_C|+2e(G[D])}{\lceil k/100\rceil}
 \le4\lceil k/100\rceil+\frac k{50}<q
\end{equation*}
for sufficiently large \(k\).  Let \(\mathcal T_D\) be the resulting family and \(U_D\) its set of used crossing edges.
Thus \(\mathcal T:=\mathcal T_C\cup\mathcal T_D\) covers exactly \(E(G[C])\cup E(G[D])\).
With \(U:=U_C\cup U_D\), the second application gives
\begin{equation*}
 \max_{c\in C}d_{U}(c)\le2\lceil k/100\rceil,
 \qquad \max_{d\in D}d_{U}(d)\le2\lceil k/100\rceil.
\end{equation*}
Together with the local bound defining \(C\cup D\), this shows that
\begin{equation*}
 \cost(\mathcal T)=|U|,
 \qquad
 \Delta\bigl(M[C,D]\cup U\bigr)< z:=3\lceil k/100\rceil.
\end{equation*}
The two applications use disjoint crossing edges, and their internal edge sets lie in opposite sides.

For all sufficiently large \(k\),
\begin{equation*}
  2z+8N
  =6\lceil k/100\rceil+8N
  \le0.38k+14<q.
\end{equation*}
Apply Lemma~\ref{dec:lem:exceptional-extension} to \(R,C,D\), treating the
family \(\mathcal T\) as previously selected.  Those triangles cover exactly
\(E(G[C])\cup E(G[D])\); the exceptional-vertex construction covers every
edge of \(G[R]\), every edge between \(R\) and \(C\cup D\) and
the remaining present \(C\)--\(D\) edges.  Thus every edge of \(G\) is covered exactly once.
\proofcase{dec:case:dense-side}
  {\(G[B]\) is dense}
Suppose \(e(G[B])>2\beta k^2\).  Define
\begin{equation*}
 X_0=\{v\in A:d_{G[A]}(v)+d_M(v)\ge\e k\},\qquad
 Y_0=\{v\in B:d_M(v)\ge\e k\}.
\end{equation*}
The sum of the quantities in these two degree conditions is at most
\begin{equation*}
  2e(G[A])+2|M|\le2\xi k^2.
\end{equation*}
Consequently,
\begin{equation*}
  |X_0|+|Y_0|\le\frac{2\xi k^2}{\e k}\le\frac{\e^2k}{2}.
\end{equation*}
Choose the largest integer \(q\) for which there are \(q\)-sets
\(C\subseteq A\setminus X_0\) and \(D\subseteq B\setminus Y_0\), and
put \(X=A\setminus C\), \(Y=B\setminus D\), and \(N=|X|+|Y|=n-2q\).
The preceding bound implies that, for large \(k\),
\begin{equation}\label{dec:eq:3.44}
 q\ge(1-\e^2)k,\qquad
 \max\{|X|,|Y|\}\le\e^2k,\qquad N\le2\e^2k.
\end{equation}

For \(c\in C\) and \(d\in D\), the definitions of \(X_0,Y_0\) give
\begin{equation}\label{dec:eq:dense control}
 d_{G[C]}(c)+d_{M[C,D]}(c)<\e k,
 \qquad d_{M[C,D]}(d)<\e k.
\end{equation}
Apply Lemma~\ref{dec:lem:triangle-extension} to \(G[C]\), with
\(M[C,D]\) forbidden, no previously used crossing edges, and the two parameters
\(\lceil\e k\rceil-1\) and \(\lceil\e k\rceil-3\).  Its numerical condition follows from
\begin{equation*}
 2(\lceil\e k\rceil-1)
 +\frac{2e(G[C])}{\lceil\e k\rceil-2}
 <2\e k+\frac{4\xi k^2}{\e k}<q,
\end{equation*}
where we used \eqref{dec:eq:3.38} and \eqref{dec:eq:3.44}.  Let \(\mathcal T_C\) be the resulting
triangle family and \(U_C\) its set of used \(C\)--\(D\) edges.  The degree
conclusions of Lemma~\ref{dec:lem:triangle-extension}, together with \eqref{dec:eq:dense control}, give
\begin{equation}\label{dec:eq:3.45}
 \cost(\mathcal T_C)=|U_C|,
 \qquad
 \Delta_{C\cup D}\bigl(M[C,D]\cup U_C\bigr)< z:=2\e k.
\end{equation}

Apply Lemma~\ref{dec:lem:exceptional-extension} to the exceptional set
\(R=X\cup Y\), with \(z\) as fixed in \eqref{dec:eq:3.45}, because
\begin{equation*}
  q-(2z+8N)
  \ge(1-4\e-17\e^2)k>0.
\end{equation*}
In implementing its construction, postpone step~\ref{dec:step:h} that retains
every still-unused present \(C\)--\(D\) edge as a \(2\)-clique.

The exceptional-vertex construction uses a \(C\)--\(D\) edge set of maximum
degree at most \(2N\): the \(K_4\)'s selected in step~\ref{dec:step:c}
use a set of maximum degree \(N\), and the matchings used in steps~\ref{dec:step:e}--\ref{dec:step:g}
for its \(N\) vertices each add at most one edge at a given vertex of
\(C\cup D\).  Let \(U_R\) be the \(C\)--\(D\) edges used by cliques 
in the exceptional-vertex construction other than the final \(2\)-cliques.
From \eqref{dec:eq:3.44} and \eqref{dec:eq:3.45},
\begin{equation*}
  \Delta_D(M[C,D]\cup U_C\cup U_R)
  <2\e k+2N<3\e q
\end{equation*}
for sufficiently large \(k\). Deleting \(Y\) from \(G[B]\) removes at most \(|Y||B|\) edges.
By \eqref{dec:eq:3.44} and \(e(G[B])>2\beta k^2\), we have
\begin{equation}\label{dec:eq:3.48}
  e(G[D])
  >(2\beta-2\e^2)k^2>\beta q^2.
\end{equation}
Apply Lemma~\ref{dec:lem:dense-completion} to
\(G[D]\) and the forbidden crossing pairs \(M[C,D]\cup U_C\cup U_R\).  It partitions
\(E(G[D])\) into triangles and \(K_4\)'s with one vertex in \(C\), using
previously unused crossing edges. For the cost calculation, imagine completing
the exceptional-vertex construction by first retaining all remaining
\(C\)--\(D\) edges as \(2\)-cliques.  Each clique supplied by
Lemma~\ref{dec:lem:dense-completion} then replaces exactly the \(2\)-cliques on the crossing
edges it uses and has the same total cost.  Thus exact coverage and the cost
are preserved.  Finally, retain every other unused present \(C\)--\(D\) edge as a \(2\)-clique.

The sparse-side triangle family, exceptional-vertex family, dense-side family,
and final \(2\)-cliques are edge-disjoint and exhaust \(E(G)\).

In both Case~\ref{dec:case:both-sparse} and Case~\ref{dec:case:dense-side}, let \(R=S\dotcupunion T\) and
\begin{equation*}
 P=C\dotcupunion S,\qquad Q=D\dotcupunion T,
\end{equation*}
be the partition supplied by Lemma~\ref{dec:lem:exceptional-extension}. Thus
\begin{equation*}
 |C|=|D|=q,\qquad |S|+|T|=N,\qquad n=2q+N.
\end{equation*}
Retain the notation from the exact identity \eqref{dec:eq:3.34}:
\(m\) is the number of missing pairs between \(R\) and
\(C\cup D\) that cross the final bipartition, \(r\) is the number of internal
\(2\)-cliques bounded by \eqref{dec:eq:add vs save}, and \(s\) is the number
of \(K_4\)'s whose exceptional vertices lie on opposite sides.

In either case the sparse-side(s) triangle family has cost equal to the number of
crossing edges it uses.  In the dense case, every clique supplied by
Lemma~\ref{dec:lem:dense-completion} has cost equal to the number of crossing-edge
\(2\)-cliques it replaces.  These replacements preserve
the exact identity \eqref{dec:eq:3.34}, which therefore holds for the final decomposition \(\mathcal D\).

Define
\begin{equation*}
\Omega:=(\lfloor \frac{n^2}{4}\rfloor-|P||Q|)+|M[C,D]|+(m-r)+m_G[S,T]+s.
\end{equation*}
Then \eqref{dec:eq:3.34} gives
\begin{equation*}
\begin{split}
 \cost(\mathcal D)
 &=(q^2-|M[C,D]|)+(qN-m)+e_{G[R]}(S,T)+r-s\\
 &=(q^2+qN+|S||T|)-\left[|M[C,D]|+(m-r)+\left(|S||T|-e_{G[R]}(S,T)\right)+s\right]\\
 &=|P||Q|-\left(|M[C,D]|+(m-r)+m_G[S,T]+s\right)\\
 &=\lfloor \frac{n^2}{4}\rfloor-\Omega.
\end{split}
\end{equation*}
Every summand in \(\Omega\) is a nonnegative integer. This proves \eqref{dec:eq:3.36}.

It remains to identify the equality case. Suppose that \(\Omega=0\). Then
\begin{equation*}
 \lfloor \frac{n^2}{4}\rfloor-|P||Q|=|M[C,D]|=m_G[S,T]=s=0,\qquad m=r.
\end{equation*}
In particular, \(P,Q\) are balanced, and \(G[C,D]\), \(G[S,T]\) are complete. By \eqref{dec:eq:add vs save},
\(m=r\) only when \(R_1=\varnothing\) and thus \(r=0\) by definition.
So \(m\) is also \(0\) and every crossing pair between \(R\) and \(C\cup D\) is present. Consequently
\begin{equation*}
 G[P,Q]=K_{|P|,|Q|},
\end{equation*}
which proves the equality assertion of the theorem.

We also record two consequences of equality needed in Section~\ref{dec:sec:strict-stability}.
First, \(R_1=\varnothing\) in this case as proved above. Moreover, all vertices of \(R_3\) lie on one side:
otherwise an edge between \(R_3\cap S\) and \(R_3\cap T\), present because \(G[S,T]\) is complete,
would be covered by a \(K_4\) whose exceptional vertices lie on opposite sides and would contribute to \(s\).
\end{proof}

\section{Strict inequality outside \texorpdfstring{\(\mathfrak E_n\)}{En} and proof of Theorem~\ref{dec:thm:main}}\label{dec:sec:strict-stability}

We first characterize graphs in \(\mathfrak E_n\).

\begin{lemma}
  \label{dec:lem:join-normal-forms}
Suppose \( G[P,Q]=K_{|P|,|Q|}\).
\begin{enumerate}[label=\textup{(\roman*)}]
\item If \(|P|=|Q|=k\), then \(G\in\mathfrak E_{2k}\) if and only if
      \(G[P]\) or \(G[Q]\) is edgeless.
\item If \(|P|=k\) and \(|Q|=k+1\), then
      \(G\in\mathfrak E_{2k+1}\) if and only if \(G[P]\) is edgeless,
      \(G[Q]\) is edgeless, or \(G[Q]\) is the spanning star \(K_{1,k}\).
\end{enumerate}
\end{lemma}

\begin{proof}
An independent set cannot meet both \(P\) and \(Q\).  This immediately
gives (i).  In (ii), an independent side of order \(k+1\) must be \(Q\).
One of order \(k\) is either \(P\) or \(Q-\{v\}\) for some \(v\in Q\).
The latter is complete to its complement exactly when \(Q-\{v\}\) is
independent and \(v\) is adjacent to all its vertices, which says that
\(G[Q]\) is the spanning star whose vertex of degree \(k\) is \(v\).
\end{proof}

If the construction in Theorem~\ref{dec:thm:stable-decomposition} gives \(\Omega =0\),
then the bipartition \(P\dotcupunion Q\) is balanced and \(G[P,Q]\) is complete.
Suppose  \(G\) is not in \(\mathfrak{E}_n\), then Lemma~\ref{dec:lem:join-normal-forms} says
both sides contain an edge. Thus there is a \(K_4\) with two vertices on each side.
We include this \(K_4\) in the decomposition, delete its edges and decompose the rest of the edges.
This new decomposition would have total cost \(\lfloor \frac{n^2}{4}\rfloor-1\).

\begin{lemma}
\label{dec:lem:strict-equality}
After increasing \(k^\ast\) if necessary, assume the hypotheses of Theorem~\ref{dec:thm:stable-decomposition},
and fix a decomposition and the associated sets \(C,D,R,R_1,R_2,R_3,S,T,P,Q\) and integer \(\Omega\) produced by that theorem.  If
\begin{equation*}
 \Omega=0
 \qquad\text{and}\qquad
 G\notin\mathfrak E_{n},
\end{equation*}
then \(G\) has an exact clique decomposition of cost \(\lfloor \frac{n^2}{4}\rfloor-1\).
\end{lemma}

\begin{proof}
Write the complete balanced final bipartition supplied by Theorem
\ref{dec:thm:stable-decomposition} as
\begin{equation*}
 P=C\dotcupunion S,\qquad Q=D\dotcupunion T,\qquad |C|=|D|=q.
\end{equation*}
Then \(N=|S|+|T|\) and
\begin{equation}\label{dec:eq:4.1}
 \bigl||S|-|T|\bigr|\le1.
\end{equation}
Let \(z\) be the bound used in the relevant case:
\begin{equation*}
 z=3\lceil k/100\rceil
\quad\text{in Case~\ref{dec:case:both-sparse}},\qquad
 z=2\e k
\quad\text{in Case~\ref{dec:case:dense-side}},
\end{equation*}
Since \(G[P,Q]\) is complete, every
vertex of \(S\) has all \(q\) neighbors in \(D\), and every vertex of
\(T\) has all \(q\) neighbors in \(C\). The equality
case of Theorem~\ref{dec:thm:stable-decomposition} says that
\(R_1=\varnothing\) and that all vertices of \(R_3\)
lie on one side when \(N\ge2\).  For \(N\le1\) the latter assertion is automatic.  Interchanging
the two sides if necessary, assume \(R_3\subseteq S\).

By \eqref{dec:eq:4.1}, every vertex of \(P\) has internal degree at most
\(|P|-1\le|Q|\). Put
\begin{equation*}
 S'=\{v\in S:d_{G[P]}(v)=|Q|\}.
\end{equation*}
A member of \(S'\) can exist only when \(|S|=|T|+1\), and it is then universal in \(G[P]\).

Lemma~\ref{dec:lem:join-normal-forms} says that both \(G[P]\) and \(G[Q]\)
contain an edge.  Choose any \(e_Q\in E(G[Q])\).  If \(G[P\setminus S']\)
contains an edge, choose \(e_P\in E(G[P\setminus S'])\).  Otherwise every
edge of \(G[P]\) meets \(S'\).  In this case
\(|S'|\ne1\): if \(S'=\{u\}\), then the universality of
\(u\) and the absence of edges in \(P-u\) make \(G[P]\) the spanning star \(K_{1,k}\),
contrary to Lemma~\ref{dec:lem:join-normal-forms}.  Choose distinct
\(u_1,v_1\in S'\) and a vertex \(c_1\in C\), put \(e_P=u_1c_1\).

Because \(G[P,Q]\) is complete, the endpoints of \(e_P,e_Q\) span a
\(K_4\), denoted by \(K'\), with two vertices in each side.  Include
\(K'\) in the decomposition and exclude all six of its edges from every subsequent selection.

Repeat the packing of the internal edges of \(C,D\) with these omissions.
A \(K'\) edge inside \(C\) or \(D\) is deleted from the corresponding
internal-edge graph, and every \(C\)--\(D\) edge of \(K'\) is declared forbidden.
Let \(U'\) be the \(C\)--\(D\) edges forbidden by \(K'\), then
\(|U'|\le4\) and \(\Delta_{C\cup D}(U')\le2\).  In the sparse case, apply
Lemma~\ref{dec:lem:triangle-extension} first with
\begin{equation*}
 (a,b)=(\lceil k/100\rceil+1,\lceil k/100\rceil-1),\qquad U=U',
\end{equation*}
and then, with the roles of \(C,D\) reversed, with
\begin{equation*}
 (a,b)=(2\lceil k/100\rceil+2,\lceil k/100\rceil-1),\qquad U=U'\cup U'_C,
\end{equation*}
where \(U'_C\) is the set of crossing edges used in the first application.
The estimates used in \eqref{dec:eq:3.3} change as follows:
\begin{equation*}
 4\lceil k/100\rceil+4+\frac{4+2\cdot10^{-4}k^2}{\lceil k/100\rceil}<q
\end{equation*}
for sufficiently large \(k\), and this common upper bound dominates the right side in
both applications.  Formula \eqref{dec:eq:3.4} gives final maximum degree at most
\(2\lceil k/100\rceil\).
In the dense case, apply Lemma~\ref{dec:lem:triangle-extension} with \(a=\lceil\e k\rceil+1, b=\lceil\e k\rceil-3\) and \(U=U'\); here
\begin{equation*}
 2\lceil\e k\rceil+2+
 \frac{4+2\xi k^2}{\lceil\e k\rceil-2}
 <q
\end{equation*}
for sufficiently large \(k\), and \eqref{dec:eq:3.4} gives maximum degree below \(\e k\).
In the sparse case, all unselected internal edges of \(C\) and \(D\) are covered.
In the dense case, all unselected internal edges of the original sparse set
among \(C,D\) are covered, and the internal edges of the original dense set
are left for Lemma~\ref{dec:lem:dense-completion}.  Every selected triangle
contributes zero to the cost with respect to the bipartition \((P,Q)\), and
the resulting graph \(F\) of forbidden \(C\)--\(D\) pairs has maximum degree
\begin{equation}\label{dec:eq:4.4}
 \Delta_{C\cup D}(F)<z.
\end{equation}

If \(e_P\) was chosen incident to \(S'\), choose \(c_2\in C\setminus\{c_1\}\) and
then choose \(d_2\in D\setminus V(e_Q)\) so that \(c_2d_2\notin F\).
This is possible since fewer than \(z\) \(C-D\) edges at \(c_2\) are forbidden.
Since \(u_1,v_1\) are universal in \(G[P]\), the vertices
\(u_1,v_1,c_2,d_2\) span a \(K_4\); denote it by \(K''\).  It is edge-disjoint
from \(K'\): its three internal edges differ from \(e_P=u_1c_1\), and
\(d_2\notin V(e_Q)\) separates their crossing edges.  Include \(K''\) in
the decomposition, exclude all six of its edges from subsequent choices,
and declare \(c_2d_2\) forbidden. Let \(\mathcal E=E(K')\cup E(K'')\).
If \(e_P\) was chosen from \(G[P\setminus S']\), let \(\mathcal E=E(K')\).

For every \(a\in S\),
\begin{equation}\label{dec:eq:4.5}
 d_{G[P]-\mathcal E}(a)
 \le d_{G[P,Q]-\mathcal E}(a).
\end{equation}
Indeed, before the reserved cliques are removed, the inequality follows
from \(d_{G[P]}(a)\le|Q|\).  If \(a\) is an exceptional endpoint of
\(e_P\), then \(K'\) removes one incident internal edge and two
incident crossing edges, decreasing the available slack by one.  In the
first case \(e_P\) avoids \(S'\), so that slack was at least one.
In the second case only \(u_1\in S'\) is an exceptional endpoint
of \(e_P\), and \(K''\) removes two further internal edges at \(u_1\) while
using one further crossing edge.  This restores the lost unit.  At \(v_1\),
\(K''\) removes two internal edges while using one crossing edge, and all
other vertices are unaffected.

We now cover all remaining edges of \(G[S\cup T]\) and all remaining edges between
\(S\cup T\) and \(C\cup D\) using cliques with zero
contribution to the cost with respect to the bipartition \((P,Q)\).  Including
the edge \(c_2d_2\) when \(K''\) is present, the \(C\)--\(D\) edges used in
this completion form a graph of maximum degree at most \(2N\).

First cover every remaining edge of \(G[S]\) by a triangle whose third
vertex lies in \(D\), choosing the third vertices greedily so that all
incident edges are unused.  Do the same for every remaining edge of
\(G[T]\), with third vertices in \(C\).  For an edge being processed, fewer than
\(2N+2<q\) possible third vertices are excluded by earlier choices and
the reserved cliques \(K'\) and \(K''\)(if present).  Thus the choices exist.

Fix \(a\in S\).  Let
 \begin{equation*}
 \begin{aligned}
  C_a&=\{c\in C:ac\in E(G)\text{ is still unused}\},\\
  D_a&=\{d\in D:ad\in E(G)\text{ is still unused}\},\\
  T_a&=\{t\in T:at\in E(G)\text{ is still unused}\}.
 \end{aligned}
 \end{equation*}
The preceding triangles remove one internal edge and use one crossing edge
at \(a\) for every edge of \(G[S]\) that they cover.  Hence \eqref{dec:eq:4.5} gives
\begin{equation*}
 |C_a|\le |D_a|+|T_a|.
\end{equation*}
Put \(r_a=\max\{0,|C_a|-|D_a|\}\), so that \(r_a\le|T_a|\).
If \(r_a>0\), choose distinct \(t_1,\dots,t_{r_a}\in T_a\) and distinct
\(c_1,\dots,c_{r_a}\in C_a\) such that every \(c_it_i\) is unused,
and use the triangle \(ac_it_i\).

These choices can be made greedily.  Indeed, if \(r_a>0\),
then the first round and the reserved cliques leave
\begin{equation*}
 |D_a|\ge q-(N-1)-3=q-N-2,
\end{equation*}
so \(|C_a|\ge q-N-1\). At any chosen \(t_i\), fewer than \(2N\) edges into \(C\) have
been used by the first round, \(K'\), and the triangles for earlier vertices of \(S\).
Excluding the fewer than \(N\) vertices \(c_j\) already chosen for \(a\) still leaves
a choice, because \(q-N-1>3N\).

After these triangles, at most \(|D_a|\) edges from \(a\) into
\(C\) remain. Match their endpoints in \(C\) into \(D_a\) through
currently unused \(C\)--\(D\) pairs. Before this matching, the forbidden
\(C\)--\(D\) graph has maximum degree at most \(z+N+1\):
\(F\) contributes at most \(z\), the edge \(c_2d_2\) contributes at most
one when \(K''\) is present, and the matching for each earlier vertex of
\(S\) contributes at most one at every vertex in \(C\).  Moreover,
\begin{equation*}
 |D_a|\ge q-N-2>2(z+N+1).
\end{equation*}
Lemma~\ref{dec:lem:hall} therefore supplies the required matching.
For an edge \(cd\) from the matching, use the triangle \(acd\), add \(cd\) to the
forbidden graph, and process all vertices of \(S\) in this way.

It remains to cover the unused edges from \(T\) into \(D\).  Fix \(b\in T\),
and let \(D_b\subseteq D\) be its set of remaining neighbors there and
let \(C_b\subseteq C\) be the set of vertices for which \(bc\) is unused.
The first round consumes at most \(|T|-1\) edges from \(b\) into \(C\), the
triangles connecting \(S,T\) and \(C\) above consume at most one for each
vertex of \(S\), and \(K'\) consumes at most two.  Consequently
\begin{equation*}
 |C_b|\ge q-(|T|-1)-|S|-2=q-N-1.
\end{equation*}
Since we assumed \(R_3\subseteq S\), \(b\) is in \(R_2\). Thus by definition \eqref{dec:eq:def R1R2R3},
\begin{equation*}
 |D_b|\le q-N-1\le|C_b|.
\end{equation*}
The currently forbidden \(C\)--\(D\) graph again has maximum degree at
most \(z+N+1\): in addition to \(F\) and possibly \(c_2d_2\), each completed
matching for a vertex of \(S\cup T\) contributes at most one.  Since
\(|C_b|\ge q-N-1>2(z+N+1)\), Lemma~\ref{dec:lem:hall}, with the roles of
\(C,D\) reversed, matches \(D_b\) into \(C_b\).  Use the corresponding
triangles and process every vertex of \(T\).

The non-matching steps use no \(C\)--\(D\) edge.  Each of the \(N\)
Hall matchings contributes a matching in \(C\times D\).  If \(K''\) is
present, its single \(C\)--\(D\) edge contributes one more; in that case
\(N\ge2\).  Hence the \(C\)--\(D\) graph used by now has maximum
degree at most \(N\) without \(K''\), and at most \(N+1\) with \(K''\).

Finally retain every still-unused edge crossing \((P,Q)\) and incident with
\(S\cup T\) as a \(2\)-clique.  All internal edges incident with
\(S\cup T\) have been covered, and every edge of \(K'\cup K''\) was
omitted from the construction, so the coverage of \(E(G[S\cup T])\) and \(E(G[S\cup T, C\cup D])\) is
exact. Every triangle and \(K''\) has zero contribution to the cost with respect to the bipartition \((P,Q)\).

In the sparse case, the triangle packing on \(C,D\) and above construction cover every
internal edge outside \(K'\cup K''\). In the dense case, let
\(L_{\rm sp},L_{\rm de}\in\{C,D\}\) denote the original sparse and dense
sets, independently of the later interchange of the two sides.  If
the selected edge on the side containing \(L_{\rm de}\)
lies inside \(L_{\rm de}\), call it \(e_{\rm de}\) and put
\(H=G[L_{\rm de}]-e_{\rm de}\); otherwise put \(H=G[L_{\rm de}]\).
The strengthened estimate \eqref{dec:eq:3.48} gives
\begin{equation*}
 e(H)\ge\beta q^2
\end{equation*}
for sufficiently large \(q\).  Let \(F'\) consist of the crossing edges used in the earlier triangle packing,
the \(C\)--\(D\) edges of \(K'\), and the \(C\)--\(D\) edges used in the
exceptional-set coverage, including the \(C\)--\(D\) edge of \(K''\) when it is
present. By \eqref{dec:eq:4.4} and the exceptional-set completion claim,
\begin{equation*}
 \Delta_{C\cup D}(F')
 \le z+(N+1)
 \le2\e k+2\e^2k+1
 <3\e q
\end{equation*}
for sufficiently large \(k\). Apply Lemma~\ref{dec:lem:dense-completion} with parameter \(\e\), taking
\(C=L_{\rm sp}\), \(D=L_{\rm de}\), and \(H\) as defined above. It covers
\(E(H)\) using previously unused edges and crossing cliques that contribute
zero to the cost with respect to the bipartition \((P,Q)\).

We have now covered every internal edge outside \(K'\cup K''\) by a
crossing clique that contributes zero to the cost with respect to the bipartition \((P,Q)\).
The clique \(K''\), when present, has three vertices in \(P\) and one in \(Q\),
so it also contributes zero, while \(K'\) covers exactly the two selected
internal edges and its four edges crossing the final bipartition.  Retain
every unused crossing edge as a \(2\)-clique. The unused-edge conditions
and the explicit omissions show that this is an exact clique decomposition of \(E(G)\);
denote the resulting family by \(\mathcal D\). It has no clique contained
in one side; every crossing clique except \(K'\) contributes zero to \(\Sigma(\mathcal D_{P,Q})\),
while \(K'\) contributes one. Equation~\eqref{dec:eq:1.2} therefore gives
\begin{equation*}
 \cost(\mathcal D)
 =e_G(P,Q)-1
 =|P||Q|-1
 =\lfloor \frac{n^2}{4}\rfloor-1.
\end{equation*}
\end{proof}

\begin{corollary}
\label{dec:cor:strict-stable}
After increasing \(k^\ast\) if necessary, every graph satisfying the
hypotheses of Theorem~\ref{dec:thm:stable-decomposition} has an exact clique
decomposition of cost at most \(\lfloor \frac{n^2}{4}\rfloor\).
Moreover, either \(G\in\mathfrak E_{n}\) or \(G\) has an exact clique decomposition
of cost at most \(\lfloor \frac{n^2}{4}\rfloor-1\).
\end{corollary}

\begin{proof}
Apply the construction in Theorem~\ref{dec:thm:stable-decomposition}. 
If its slack satisfies \(\Omega\ge1\), then \eqref{dec:eq:3.36} is already strict.
If \(\Omega=0\) and \(G\notin\mathfrak E_{n}\), apply Lemma~\ref{dec:lem:strict-equality}.
In the remaining case \(G\in\mathfrak E_{n}\), \eqref{dec:eq:3.36} supplies the required upper bound.
\end{proof}

We are now ready to prove Theorem~\ref{dec:thm:main}. We will use a special case of
Theorem~1.10 of~\cite{BHKNW} which says that \(\CD_2(G,(i-1))\) is not far
from \(\CDstar_2(G,(i-1))\) for large \(n\).
\begin{theorem}
  \label{dec:thm:BHKNW}
For every fixed \(\rho>0\), there is \(n_0=n_0(\rho)\) such that every
graph \(G\) of order \(n\ge n_0\) satisfies
\begin{equation*}
  \CD_2(G,(i-1))\le \CDstar_2(G,(i-1))+\rho n^2.
\end{equation*}
\end{theorem}

\begin{proof}[Proof of Theorem~\ref{dec:thm:main}]
Let \(\xi\) be supplied by Theorem~\ref{dec:thm:stable-decomposition}, choose
\(k^\ast\) so that Corollary~\ref{dec:cor:strict-stable} holds, and put \(\eta=\xi/5\).
Apply Theorem~\ref{dec:thm:fractional-stability} with this fixed \(\eta\),
and let \(\delta>0,n_0\) be the resulting constants.
Apply Theorem~\ref{dec:thm:BHKNW} with \(\rho=\delta/2\).
Choose \(n^\ast\) above all resulting fixed thresholds and large enough that
\(n=2k+\sigma\), \(\sigma\in\{0,1\}\), implies \(k\ge\max\{5,k^\ast\}\).

Let \(G\) be an \(n\)-vertex graph with \(n\ge n^\ast\).
If \(\CDstar_2(G,(i-1))<\lfloor \frac{n^2}{4}\rfloor-\delta n^2\), then Theorem~\ref{dec:thm:BHKNW} gives
\(\CD_2(G,(i-1))\le\CDstar_2(G,(i-1))+\delta n^2/2<\lfloor \frac{n^2}{4}\rfloor\),
and integrality yields \(\CD_2(G,(i-1))\le \lfloor \frac{n^2}{4}\rfloor-1\).

Otherwise Theorem~\ref{dec:thm:fractional-stability} gives a balanced
partition \(S\dotcupunion T\), with the sides interchanged if necessary,
such that \(e(G[S])+m_G(S,T)\le\eta n^2\). Write \(n=2k+\sigma\),
where \(\sigma\in\{0,1\}\). Since \(n^2\le5k^2\), this is at most \(\xi k^2\).
Corollary~\ref{dec:cor:strict-stable} gives a decomposition of cost at most
\(\lfloor \frac{n^2}{4}\rfloor\), and gives one of cost at most \(\lfloor \frac{n^2}{4}\rfloor-1\) whenever
\(G\notin\mathfrak E_n\). Together with the first case, this proves
\begin{equation}\label{dec:eq:4.7}
 G\notin\mathfrak E_n
 \quad\Longrightarrow\quad
 \CD_2(G,(i-1))\le \lfloor \frac{n^2}{4}\rfloor-1.
\end{equation}

It remains to prove equality for \(G\in\mathfrak E_n\). Choose a
balanced partition \(V(G)=S\dotcupunion T\) such that \(S\) is independent and \(G[S,T]\) is complete.
If \(\mathcal D\) is any exact clique decomposition of \(G\), let \(\mathcal D_T\) be its subfamily contained in \(T\).
Every crossing clique contains exactly one vertex of \(S\), so its contribution to \(\Sigma(\mathcal D_{S,T})\) is zero,
and there is no clique contained in \(S\). Equation~\eqref{dec:eq:1.2} therefore gives
\begin{equation*}
 \cost(\mathcal D)
 =e_G(S,T)+\cost(\mathcal D_T)
 \ge |S||T|=\lfloor \frac{n^2}{4}\rfloor.
\end{equation*}
The upper bound already proved gives \(\CD_2(G,(i-1))=\lfloor \frac{n^2}{4}\rfloor\).
Combining this with \eqref{dec:eq:4.7} proves the equality classification.
\end{proof}

\section{Stability for \texorpdfstring{\(t\)}{t}-clique covers}\label{cov:sec:stability}

This section proves a stability result for \(t\)-clique covers. We first record some notation and external lemmas.
For disjoint nonempty vertex sets \(A,B\) in a graph, write \(d(A,B):=\frac{e(A,B)}{|A||B|}\).
The pair \((A,B)\) is \emph{$\varepsilon$-regular} if \(|d(A',B')-d(A,B)|<\varepsilon\)
whenever \(A'\subseteq A\) and \(B'\subseteq B\) satisfy \(|A'|\ge\varepsilon|A|\) and \(|B'|\ge\varepsilon|B|\).
An \emph{$(m,\ell,\varepsilon,d)$-graph} is a graph \(H\) on \(m\) vertices with an equipartition
\(V(H)=W_1\cup\cdots\cup W_\ell\) into independent sets such that, for every \(i<j\),
either \(E_H(W_i,W_j)=\varnothing\) or \((W_i,W_j)\) is \(\varepsilon\)-regular with density at least \(d\).
We write \(\omega(G)\) for the clique number.

The next four lemmas are the unit-cost forms of Lemmas 4.2 and 4.9--4.11 in \cite{BHKNW}.

\begin{lemma}
  \label{cov:lem:cleaning}
For every \(\varepsilon>0\) and positive integer \(m_0\),
there is \(M=M(\varepsilon,m_0)\) such that the following holds.
Let \(G\) be an \(n\)-vertex graph and let \(d>0\).
Then \(G\) has a subgraph \(\widetilde G\) on \(m\ge(1-\varepsilon)n\) vertices
which is an \((m,\ell,\varepsilon,d)\)-graph for some \(m_0<\ell\le M\), and
\begin{equation*}
|E(G)\setminus E(\widetilde G)|
 \le (2\varepsilon+m_0^{-1}+d)n^2.
\end{equation*}
\end{lemma}

\begin{lemma}
  \label{cov:lem:cleaned-rounding}
Fix \(t\ge2\).  For every integer \(b\ge t\) and every \(d,\rho\in(0,1)\),
there are \(\varepsilon_0>0\) and \(n_0\) such that the following holds.
If \(G\) is an \((n,\ell,\varepsilon,d)\)-graph with
\begin{equation*}
n\ge n_0,\qquad \omega(G)\le b,\qquad \varepsilon<\varepsilon_0,
\end{equation*}
then
\begin{equation*}
\CC_t(G)\le \CCstar_t(G)+\rho n^t.
\end{equation*}
\end{lemma}

\begin{lemma}
  \label{cov:lem:restore-edges}
Let \(G\) be an \(n\)-vertex graph, let \(t\ge2\), and let \(H\) be a subgraph of \(G\). Then
\begin{equation*}
\CC_t(G)\le \CC_t(H)+|E(G)\setminus E(H)|n^{t-2}.
\end{equation*}
\end{lemma}

\begin{lemma}
  \label{cov:lem:bounded-clique}
Let \(G\) be an \(n\)-vertex graph, and let \(b\ge2\) with \(n\ge4b\). Then \(G\) has a spanning subgraph \(H\) such that
\begin{equation*}
\omega(H)\le b
\qquad\text{and}\qquad
|E(G)\setminus E(H)|\le \frac{n^2}{b}.
\end{equation*}
\end{lemma}

Combining the preceding lemmas gives a spanning subgraph \(G_0\subseteq G\) with
very few edges deleted and \(\CCstar_t(G_0)\) close to \(\CC_t(G)\).
So asymptotically we only need to consider fractional \(t\)-clique covers.

\begin{lemma}\label{cov:lem:nearby-model}
Fix \(t\ge 2\).  For every \(\eta>0\), there is \(N(t,\eta)\) such that every
graph \(G\) on \(n\ge N(t,\eta)\) vertices has a spanning subgraph \(G_0\) satisfying
\begin{equation*}
|E(G)\setminus E(G_0)|\le \eta n^2
\end{equation*}
and
\begin{equation*}
\CC_t(G)\le \CCstar_t(G_0)+\eta n^t.
\end{equation*}
\end{lemma}

\begin{proof}
It suffices to treat \(0<\eta<1\). Choose an integer \(b\ge\max\{t,4/\eta\}\),
then choose \(d>0\) and a positive integer \(m_0\) so that \(b^{-1}+m_0^{-1}+d<\eta/2\).
Let \(\e_0\) and \(n_0\) be supplied by Lemma~\ref{cov:lem:cleaned-rounding} with parameters \(b,d\) and
\(\rho=\eta/4\), and choose \(0<\varepsilon<\varepsilon_0\) so small that
\begin{equation*}
b^{-1}+2\varepsilon+m_0^{-1}+d\le\frac{\eta}{2}.
\end{equation*}
For sufficiently large \(n\), Lemmas~\ref{cov:lem:bounded-clique} and
\ref{cov:lem:cleaning} give a cleaned subgraph \(\widetilde G\) of order \(m\) at
least \((1-\e)n\) and clique number at most \(b\) which is an \((m,\ell,\e,d)\)-graph.
Add its omitted vertices back as isolated vertices to obtain a spanning graph \(G_0\).
The deletion estimates give
\begin{equation*}
 |E(G)\setminus E(G_0)|
 \le (b^{-1}+2\varepsilon+m_0^{-1}+d)n^2
 \le\frac{\eta}{2}n^2.
\end{equation*}
Isolated vertices do not affect either cover parameter,
so Lemma~\ref{cov:lem:restore-edges} and then Lemma~\ref{cov:lem:cleaned-rounding} yield
\begin{equation*}
 \CC_t(G)\le\CC_t(G_0)+\frac{\eta}{2}n^t
 \le\CCstar_t(G_0)+\frac{3\eta}{4}n^t.
\end{equation*}
\end{proof}

Recall that in a graph \(G\), two vertices \(u,v\) are \emph{clones} of each other if \(N_G(u)=N_G(v)\).
This is an equivalence relation; its classes are called \emph{clone classes}, and each
clone class is independent. Let \(A\) and \(B\) be distinct clone classes with no edges between them.
The \emph{symmetrization of \(B\) to \(A\)}, denoted by \(G^{B\to A}\), deletes all
edges incident to \(B\) and then joins every vertex of \(B\) to
the common neighborhood of the vertices of \(A\); all other adjacencies
remain unchanged. Thus \(A\cup B\) becomes one clone class.  Define \(G^{A\to B}\)
analogously by symmetrizing \(A\) to \(B\). The next lemma is the weighted inequality
established in the proof of Lemma 2.1 in \cite{BHKNW}.

\begin{lemma}
  \label{cov:lem:weighted-symmetrization}
With the notation above,
\begin{equation*}
\CCstar_t(G)\le \frac{|A|}{|A|+|B|}\CCstar_t(G^{B\to A})+\frac{|B|}{|A|+|B|}\CCstar_t(G^{A\to B}).
\end{equation*}
\end{lemma}

Inspired by this inequality, we define a symmetrization random process as follows.
Starting from \(G\), choose \(G^{B\to A}\) with probability \(|A|/(|A|+|B|)\) and \(G^{A\to B}\)
with probability \(|B|/(|A|+|B|)\). Repeat this operation whenever there are two nonadjacent clone classes.
Use any fixed rule to select the next pair. Then the expectation of fractional \(t\)-clique cover number
is nondecreasing, and the number of \(s\)-cliques is a martingale for any \(s\ge2\).

\begin{lemma}\label{cov:lem:random-symmetrization}
The process terminates at a complete multipartite graph \(\mathbf G\).  Moreover,
\begin{equation*}
\CCstar_t(G)\le \mathbb E \CCstar_t(\mathbf G)
\end{equation*}
and, for every \(s\ge 2\),
\begin{equation*}
|\mathcal{K}_s(G)|=\mathbb E |\mathcal K_s(\mathbf G)|.
\end{equation*}
In particular if \(G\) is \(K_s\)-free, then every graph produced by the process is \(K_s\) free.
\end{lemma}

\begin{proof}
At each step, the two selected clone classes merge, and no clone class splits.
Thus the number of clone classes decreases.  At termination, every two distinct clone classes
are completely adjacent, so the terminal graph is complete multipartite.

The first expectation inequality follows by iterating Lemma~\ref{cov:lem:weighted-symmetrization}.
Since all vertices in a clone class have the same neighborhood, let \(N_G(A)\) and \(N_G(B)\) denote
the common neighborhoods of the vertices in \(A\) and \(B\) respectively.  For the clique count, let
\begin{equation*}
c_A:=|\mathcal K_{s-1}(G[N_G(A)])|
\quad\text{and}\quad
c_B:=|\mathcal K_{s-1}(G[N_G(B)])|.
\end{equation*}
Since \(A\cup B\) is independent, the number of copies of \(K_s\) meeting \(A\cup B\) is \(|A|c_A+|B|c_B\).
After the two possible symmetrizations, this number is respectively \((|A|+|B|)c_A\) and \((|A|+|B|)c_B\).
Its conditional expectation is therefore unchanged.  Copies avoiding \(A\cup B\) are unchanged.
Iterating the conditional identity proves the martingale formula.

Finally, suppose that a symmetrized graph contains a new \(K_s\) using a vertex
in the class that was cloned. Replacing that vertex by a vertex of the clone class
whose neighborhood was copied gives a \(K_s\) in the preceding graph.
Hence symmetrization preserves \(K_s\)-freeness.
\end{proof}

Since the symmetrization random process always terminates at a complete multipartite graph,
we'd like to calculate the fractional \(t\)-clique cover number in such case, which is straightforward.

\begin{lemma}\label{cov:lem:multipartite-value}
Let \(G\) be complete multipartite with nonzero part sizes
\begin{equation*}
a_1\ge a_2\ge\cdots\ge a_k,
\end{equation*}
and append zero parts when \(k<t\).  Then
\begin{equation*}
\CCstar_t(G)=\prod_{i=1}^t a_i.
\end{equation*}
\end{lemma}

\begin{proof}
If \(k<t\), both sides are zero. Assume \(k\ge t\).
There are \(\prod_{i=1}^t a_i\) copies of \(K_t\) using the first \(t\) parts.
A clique of \(G\) contains at most one of these copies.
Summing their fractional-cover constraints gives
\begin{equation*}
\CCstar_t(G)\ge \prod_{i=1}^t a_i.
\end{equation*}

Give every maximal clique weight \(\left(\prod_{i=t+1}^k a_i\right)^{-1}\).
The total weight is \(\prod_{i=1}^t a_i\).  A copy of \(K_t\) using the parts indexed by \(I\) receives total weight
\begin{equation*}
\frac{\prod_{j\notin I}a_j}{\prod_{j>t}a_j}
 =\frac{\prod_{i=1}^t a_i}{\prod_{i\in I}a_i}
 \ge 1,
\end{equation*}
since the first \(t\) part sizes have the largest product. This proves the reverse inequality.
\end{proof}

Lemma~\ref{cov:lem:random-symmetrization} and Lemma~\ref{cov:lem:multipartite-value} imply
\(\CCstar_t(G)\le \CC_t(T_{n,t})\) for every \(n\)-vertex graph \(G\): in every terminal graph,
the sum of the \(t\) largest part sizes is at most \(n\), so their product is at most \(\CC_t(T_{n,t})\).
Moreover, Lemma~\ref{cov:lem:random-symmetrization} allows us to keep track of \(\CCstar_t(\mathbf G)\)
and \(\mathcal K_s(\mathbf G)\) at the same time. The next lemma says the number of \(t+1\)-cliques in \(G\)
is controlled by its fractional \(t\)-clique cover number.

\begin{lemma}\label{cov:lem:few-larger-cliques}
For every fixed \(t\), there is a constant \(C=C(t)\) such that every \(n\)-vertex graph \(G\) satisfies
\begin{equation*}
|\mathcal K_{t+1}(G)|
 \le \frac{t^t}{t!}\,n\bigl(\CC_t(T_{n,t})-\CCstar_t(G)\bigr)+Cn^t.
\end{equation*}
\end{lemma}

\begin{proof}
Apply Lemma~\ref{cov:lem:random-symmetrization}.  Let
\begin{equation*}
x_1\ge x_2\ge\cdots
\end{equation*}
be the normalized part sizes of the terminal graph \(\mathbf G\), and set \(r:=\sum_{i>t}x_i\).
By Lemma~\ref{cov:lem:multipartite-value} and the arithmetic--geometric mean inequality,
\begin{equation*}
\frac{\CCstar_t(\mathbf G)}{n^t}
 =\prod_{i=1}^t x_i
 \le \left(\frac{1-r}{t}\right)^t.
\end{equation*}
Since \(1-(1-r)^t\ge r\) for \(0\le r\le 1\),
\begin{equation*}
r\le
\frac{t^t}{n^t}\left(\frac{n^t}{t^t}-\CCstar_t(\mathbf G)\right).
\end{equation*}
Also \(\CC_t(T_{n,t})=n^t/t^t+O_t(n^{t-1})\), and therefore
\begin{equation}\label{cov:eq:tail-deficit}
r\le \frac{t^t}{n^t}\bigl(\CC_t(T_{n,t})-\CCstar_t(\mathbf G)\bigr)+O_t(n^{-1}).
\end{equation}

Let \(W\) be the union of the parts after the largest \(t\) parts, so
\(|W|=rn\).  Every copy of \(K_{t+1}\) meets \(W\), and therefore
\begin{equation*}
|\mathcal K_{t+1}(\mathbf G)|
 \le |W|\binom nt
 \le \frac{r}{t!}n^{t+1}.
\end{equation*}
Take expectations in this inequality and in \eqref{cov:eq:tail-deficit}. Lemma~\ref{cov:lem:random-symmetrization} gives
\begin{align*}
|\mathcal K_{t+1}(G)|
 &=\mathbb E |\mathcal K_{t+1}(\mathbf G)|\\
 &\le \frac{t^t}{t!}\,n\bigl(\CC_t(T_{n,t})-\mathbb E \CCstar_t(\mathbf G)\bigr)+O_t(n^t)\\
 &\le \frac{t^t}{t!}\,n\bigl(\CC_t(T_{n,t})-\CCstar_t(G)\bigr)+O_t(n^t).
\end{align*}
This proves the estimate.
\end{proof}

When \(\CCstar_t(G)\) is close to \(\CC_t(T_{n,t})\), Lemma~\ref{cov:lem:few-larger-cliques} and
graph removal lemma~\cite{Removal} say we can delete very few edges to make \(G\) \(K_{t+1}\)-free.
Together with a result by Ma and Qiu~\cite{MaQiu}, this proves the edit distance between \(G\) and
\(T_{n,t}\) is small. Here the edit distance between \(G\) and \(H\), denoted by \(\ed(G,H)\),
is the minimum of \(|E(G)\triangle E(H')|\) over all copies \(H'\) of \(H\) on \(V(G)\).
For completeness we record the external results we use below.

\begin{theorem}[Graph removal lemma \cite{Removal}]\label{cov:thm:graph-removal}
For every fixed graph $F$ and every $\varepsilon>0$, there is $\delta>0$
such that the following holds.  If an $n$-vertex graph $G$
contains at most $\delta n^{|V(F)|}$ copies of $F$, then one can delete at most
$\varepsilon n^2$ edges from $G$ to obtain an $F$-free graph.
\end{theorem}

\begin{theorem}[Ma and Qiu~\cite{MaQiu}]\label{cov:thm:maqiu}
  For every integer \(r\ge2\) and every \(\e>0\), there are \(\delta>0\) and \(N_0\) such that,
  if \(n\ge N_0\) and \(G\) is an \(n\)-vertex \(K_{r+1}\)-free graph with at most \(\delta n^r\)
  fewer copies of \(K_r\) than \(T_{n,r}\), then \(\ed(G,T_{n,r})\le \e n^2\).
\end{theorem}

The stability result for fractional \(t\)-clique covers follows directly.

\begin{theorem}
  \label{cov:prop:fractional-stability}
Fix \(t\ge2\). For every \(\varepsilon>0\), there are \(\delta>0\) and \(N\) such that
every graph \(G\) on \(n\ge N\) vertices satisfying
\begin{equation*}
\CCstar_t(G)\ge \CC_t(T_{n,t})-\delta n^t
\end{equation*}
also satisfies
\begin{equation*}
\ed(G,T_{n,t})\le \varepsilon n^2.
\end{equation*}
\end{theorem}

\begin{proof}

First apply Theorem~\ref{cov:thm:maqiu} with \(r=t\) and edit distance error \(\frac{\e}{2}\),
let \(\sigma\) and \(N_0\) be the resulting constants. Then apply Theorem~\ref{cov:thm:graph-removal}
to \(K_{t+1}\) with deletion error \(\eta\le\min\{\frac\sigma2,\frac\e2\}\), obtaining a copy threshold \(\tau\).
Let \(C=C(t)\) come from Lemma~\ref{cov:lem:few-larger-cliques}, and choose \(0<\delta\le\min\{\frac\sigma2,t!\tau/(2t^t)\}\).
Finally take \(N\ge N_0\) so large that, for \(n\ge N\), \(C/n\le\tau/2\).

Let \(G\) satisfy the hypotheses.  Lemma~\ref{cov:lem:few-larger-cliques} gives
\begin{equation*}
|\mathcal K_{t+1}(G)|
 \le \frac{t^t}{t!}\delta n^{t+1}+Cn^t
 \le \tau n^{t+1}.
\end{equation*}
The graph removal lemma therefore gives a \(K_{t+1}\)-free spanning subgraph \(H\subseteq G\) obtained by
deleting at most \(\eta n^2\) edges.  Since each deleted edge lies in at most \(n^{t-2}\) copies of \(K_t\),
\begin{equation*}
|\mathcal K_t(H)|
 \ge |\mathcal K_t(G)|-\eta n^t
 \ge \CCstar_t(G)-\eta n^t
 \ge \CC_t(T_{n,t})-(\delta+\eta)n^t\ge|\mathcal K_t(T_{n,t})|-\sigma n^t.
\end{equation*}
Theorem~\ref{cov:thm:maqiu} gives
\begin{equation*}
\ed(H,T_{n,t})\le \frac{\varepsilon}{2}n^2.
\end{equation*}
Since \(G\) and \(H\) differ in at most \(\eta n^2\le \varepsilon n^2/2\) edges,
\begin{equation*}
\ed(G,T_{n,t})\le \varepsilon n^2.
\end{equation*}
\end{proof}

Finally we pull this stability result back to integer \(t\)-clique covers.

\begin{corollary}
  \label{cov:cor:integer-stability}
Fix \(t\ge2\).  For every \(\varepsilon>0\), there are \(\delta>0\) and \(N\) such that
every graph \(G\) on \(n\ge N\) vertices satisfying
\begin{equation*}
\CC_t(G)\ge \CC_t(T_{n,t})-\delta n^t
\end{equation*}
also satisfies
\begin{equation*}
\ed(G,T_{n,t})\le \varepsilon n^2.
\end{equation*}
\end{corollary}

\begin{proof}
Let \(\delta_0,N_0\) be supplied by Theorem~\ref{cov:prop:fractional-stability} with error \(\varepsilon/2\). 
Set \(\eta=\min\{\delta_0/2,\varepsilon/2\}\), \(\delta=\delta_0/2\), and take
\(N\) above \(N_0\) and the threshold in Lemma~\ref{cov:lem:nearby-model}
for \(\eta\).  For the resulting subgraph \(G_0\subseteq G\),
\begin{equation*}
\CCstar_t(G_0)
 \ge \CC_t(G)-\eta n^t
 \ge \CC_t(T_{n,t})-(\delta+\eta)n^t
 \ge \CC_t(T_{n,t})-\delta_0 n^t.
\end{equation*}
Thus Theorem~\ref{cov:prop:fractional-stability} gives \(\ed(G_0,T_{n,t})\le\varepsilon n^2/2\), while
\(|E(G)\setminus E(G_0)|\le\eta n^2\le\varepsilon n^2/2\) by Lemma~\ref{cov:lem:nearby-model}. This proves the claim.
\end{proof}

\section{Proof of Theorem~\ref{cov:thm:main}}\label{cov:sec:transversal matching}

For an ordered \(t\)-partition \(\mathcal P=(V_1,\ldots,V_t)\) of \(V(G)\) (empty classes are permitted), define
\begin{equation*}
D_G(\mathcal P):=
\sum_{i=1}^t e(G[V_i])
+
\sum_{1\le i<j\le t}\bigl(|V_i||V_j|-e_G(V_i,V_j)\bigr).
\end{equation*}
This is the number of edge changes required to turn \(G\) into the complete \(t\)-partite graph with parts \(V_1,\ldots,V_t\).
The next lemma finds a good \(t\)-partition \((V_1,\ldots,V_t)\)
when \(\ed(G,T_{n,t})\) is small, such that the size of each part is roughly \(n/t\).
Denote the set of vertices in \(V_i\) with large neighborhoods inside \(V_i\) by \(X_i\).
The partition also ensures that the total size of all \(X_i\)'s is small and that each vertex
in each \(X_i\) also has many neighbors in each other \(V_j\).

\begin{lemma}\label{cov:lem:edit-partition}
Fix \(\theta>0\). Let \(G\) be an \(n\)-vertex graph satisfying \(\ed(G,T_{n,t})\le \theta n^2\),
and let \(\mathcal P=(V_1,\ldots,V_t)\) minimize \(D_G(\mathcal P)\) over all \(t\)-partitions.  Then
\begin{equation}\label{cov:eq:part-variance}
D_G(\mathcal P)\le \theta n^2, \qquad\sum_{i=1}^t\left(|V_i|-\frac nt\right)^2
 \le 4\theta n^2+\frac t4.
\end{equation}
Fix \(\gamma>0\) and define
\begin{equation*}
X_i:=\{v\in V_i:d(v,V_i)>\gamma n\},
\qquad
U_i:=V_i\setminus X_i,
\qquad
X:=\bigcup_iX_i.
\end{equation*}
Then
\begin{align}
  |X|&\le \frac{2\theta}{\gamma}n,\label{cov:eq:exceptional-size-quantitative}\\
  d(u,V_i)&\le\gamma n &&\text{for\ } i\in[t],\ u\in U_i,\notag\\
  d(x,V_j)& \ge \gamma n-\sqrt{4\theta n^2+\frac t4} &&\text{for\ }i\in[t],\ x\in X_i \text{\ and\ } j\ne i.\label{cov:eq:exceptional-cross-degree-quantitative}
\end{align}
\end{lemma}

\begin{proof}
Choose a copy of \(T_{n,t}\) on \(V(G)\) whose edit distance from \(G\) is at most \(\theta n^2\).
Its balanced \(t\)-partition \(\mathcal P'\) has \(D_G(\mathcal P')\) at most \(\theta n^2\).
Hence the partition \(\mathcal P=(V_1,\ldots,V_t)\) minimizing \(D_G\) satisfies \(D_G(\mathcal P)\le\theta n^2\).
Let \(b_1,\ldots,b_t\) be the balanced part sizes of \(T_{n,t}\), then
\begin{equation*}
  \begin{aligned}
    \sum_{i=1}^t|V_i|^2-\sum_{i=1}^t b_i^2 &= 2\left(e(T_{n,t})-\sum_{1\le i<j\le t}|V_i||V_j|\right)\\
    &= 2\left[e(T_{n,t})-\left(D_G(\mathcal P)-2\sum_{i=1}^t e(G[V_i])+e(G)\right)\right]\\
    &= 2\left[\left(e(T_{n,t})-e(G)\right)-D_G(\mathcal P)+2\sum_{i=1}^t e(G[V_i])\right]\\
    &\le 2\left(\theta n^2-D_G(\mathcal P)+2D_G(\mathcal P)\right)\\
    &\le 4\theta n^2.
  \end{aligned}
\end{equation*}
Moreover, let \(r\) be the residue of \(n\) modulo \(t\), then
\begin{equation*}
\sum_{i=1}^t b_i^2-\frac{n^2}{t}
 =\frac{r(t-r)}{t}\le\frac t4,
\end{equation*}
Thus
\begin{equation*}
  \begin{aligned}
    \sum_{i=1}^t\left(|V_i|-\frac nt\right)^2 &= \sum_{i=1}^t\left(|V_i|^2-b_i^2+b_i^2+\frac{n^2}{t^2}-2\frac{n}{t}|V_i|\right)\\
    &= \left(\sum_{i=1}^t|V_i|^2-\sum_{i=1}^t b_i^2\right)+\left(\sum_{i=1}^t b_i^2-\frac{n^2}{t}\right)\\
    &\le 4\theta n^2+\frac{t}{4}.
  \end{aligned}
\end{equation*}
This is \eqref{cov:eq:part-variance}.

Every vertex in \(X\) has more than \(\gamma n\) neighbors in its own part, while
\begin{equation*}
\sum_{i=1}^t\sum_{v\in V_i}d(v,X_i)\le2\sum_{i=1}^t e(G[V_i])\le2\theta n^2.
\end{equation*}
This proves \eqref{cov:eq:exceptional-size-quantitative}.  The degree bound on \(U_i\) follows from its definition.

Finally, by moving some \(x\in V_i\) to \(V_j\), the change in the partition edit cost is
\begin{equation*}
2\left(d(x,V_j)-d(x,V_i)\right)+|V_i|-|V_j|-1.
\end{equation*}
Minimality makes this quantity nonnegative.  Hence
\begin{equation*}
d(x,V_j)
 \ge d(x,V_i)-\frac{|V_i|-|V_j|}{2}+\frac12.
\end{equation*}
By \eqref{cov:eq:part-variance},
\begin{equation*}
\frac{|V_i|-|V_j|}{2}
 \le \sqrt{4\theta n^2+\frac t4}.
\end{equation*}
Since \(d(x,V_i)>\gamma n\), this gives \eqref{cov:eq:exceptional-cross-degree-quantitative}.
\end{proof}

For a graph \(G\) with a \(t\)-partition \(V(G)=V_1\dotcupunion\cdots\dotcupunion V_t\), let
\(\mathfrak T:=V_1\times\cdots\times V_t\) be the set of \(t\)-tuples. Recall that a \(t\)-set is
\emph{transversal} if it contains exactly one vertex from each part. We identify each tuple
in \(\mathfrak T\) with its underlying transversal set, for simplicity we may call a (non)-transversal
set a (non)-transversal. Let \(\mathcal A\subseteq\mathfrak T\) be the set of transversals that induce
\(K_t\).  A transversal in \(\mathfrak T\setminus\mathcal A\) is called \emph{missing}.

Let each transversal \(K_t\) be covered by itself. A non-transversal \(K_t\) could either be covered
by itself or by a clique enlarged from some transversal \(K_t\). In the first case, its contribution
to \(\CC_t(G)\) could be compensated by some missing transversal \(t\)-set.
We formalize this as follows: a transversal \(t\)-set \(T\) is \emph{admissible} or a \emph{candidate}
for a non-transversal copy \(S\) of \(K_t\) if either \(T\notin\mathcal A\), or \(T\in\mathcal A\) and \(S\cup T\) is a clique.
We will show that, given a good \(t\)-partition as constructed in Lemma~\ref{cov:lem:edit-partition},
we could assign distinct admissible transversals to all non-transversal copies of \(K_t\) so that
the total cost used to cover all \(K_t\)'s does not exceed \(\prod_{i=1}^t|V_i|\).

\begin{theorem}\label{cov:thm:transversal-matching}
Fix \(t\ge 2\) and \(0<\alpha\le 1/t\). There is \(\gamma=\gamma(t,\alpha)\in(0,1/2)\) such that, for every
\(\kappa>0\), there are \(\xi=\xi(t,\alpha,\kappa)\in(0,1/2)\) and \(N=N(t,\alpha,\kappa)\) with the following property.

Let \(G\) be an \(n\)-vertex graph with \(n\ge N\). Suppose that
\begin{equation*}
V(G)=V_1\dotcupunion\cdots\dotcupunion V_t,
\qquad |V_i|\ge \alpha n\quad\text{for }i\in[t].
\end{equation*}
Write $V_i=U_i\dotcupunion X_i$ and $X=\bigcup_iX_i$. Suppose \(|X|\le \xi n\), and that for \(i\in[t]\),
\begin{equation*}
  \begin{aligned}
    d(u,V_i)&\le \gamma n \qquad&&\text{for }u\in U_i,\\
    d(x,V_j)&\ge \kappa n \qquad&&\text{for }x\in X_i\text{ and }j\ne i.
  \end{aligned}
\end{equation*}
Then
\begin{equation*}
\CC_t(G)\le\prod_{i=1}^t|V_i|.
\end{equation*}
If equality holds, then $G$ is the complete $t$-partite graph with parts $V_1,\ldots,V_t$.
\end{theorem}

\begin{proof}
Suppose we have \(G\) as stated in the theorem, leave \(\gamma, \xi, N\) to be decided. 
Set \(\sigma:=\gamma+\xi<1\). Fix a linear order on every part \(V_i\).
Fix a non-transversal copy of \(K_t\) \(S\), put
\begin{equation*}
r_i:=|S\cap V_i|,
\qquad
A:=\{i:r_i>0\},
\qquad
B:=\{i:r_i=0\},
\qquad
d:=|B|.
\end{equation*}
Let \(P:=(r_1,\cdots,r_t)\) be its multiplicity pattern. Since \(|S|=t\),
\begin{equation}\label{cov:eq:excess-identity}
\sum_{i\in A}(r_i-1)=d.
\end{equation}
For every \(i\in A\), choose a canonical representative \(b_i\in S\cap V_i\):
take the first vertex of \(S\cap U_i\) if this set is nonempty, and otherwise
take the first vertex of \(S\cap X_i\). The other vertices of \(S\cap V_i\) are called \emph{nonrepresentative vertices}.

Given the canonical representatives \(b_i\)'s, the nonrepresentative vertices have at most \((\sigma n)^d\)
possible choices. Indeed, if \(b_i\in U_i\), every nonrepresentative vertex in \(V_i\) lies in \(N(b_i,V_i)\),
which has size at most \(\gamma n\). If \(b_i\in X_i\), then \(S\cap U_i=\varnothing\) by the definition of \(b_i\),
so every nonrepresentative vertex in \(V_i\) lies in \(X_i\), which has size at most \(\xi n\).
And \eqref{cov:eq:excess-identity} says there are exactly \(d\) nonrepresentative vertices to be decided.

We call \(S\) \emph{singular} if \(d=1\) and its unique repeated part \(V_i\) satisfies
\begin{equation*}
S\cap V_i=\{u,x\},
\qquad u\in U_i,
\quad x\in X_i.
\end{equation*}
Let \(V_j\) be the unique missing part. For every \(y\in N(x,V_j)\), form the transversal that
contains \(u\) in \(V_i\), \(y\) in \(V_j\), and the unique vertex of \(S\) in every other part.
This transversal is admissible for \(S\): if it is a clique, its union with \(S\) is a clique;
otherwise it is missing. Thus every singular \(S\) has at least \(\kappa n\) candidates.

A fixed transversal is produced as a candidate by this construction for at most \(t(t-1)|X|\) singular copies.
After the repeated part and the missing part are fixed, the transversal determines \(u\), the vertex chosen
in the missing part, and all singleton vertices of \(S\); only \(x\in X_i\) remains undetermined.

We now define candidates for every nonsingular \(S\). For each completion vector \(y=(y_j)_{j\in B}\) in \(\prod_{j\in B}V_j\), let
\begin{equation*}
T_0(y):=\{b_i:i\in A\}\cup\{y_j:j\in B\}.
\end{equation*}
If \(T_0(y)\) is missing, take \(T_0(y)\) as the candidate. If \(S\cup T_0(y)\) is a clique,
again take \(T_0(y)\). Otherwise both \(S\) and \(T_0(y)\) are cliques but \(S\cup T_0(y)\) is not.
Thus there exist a nonrepresentative vertex \(z\in S\cap V_i\) and a completion vertex \(y_j\) such
that \(zy_j\notin E(G)\). Order the indexed pairs lexicographically using the part indices and the
fixed vertex orders. Choose the first such pair, replace \(b_i\) by \(z\), and take the resulting transversal.
It contains the nonedge \(zy_j\), thus is missing and hence admissible.

Different completion vectors give different candidate transversals.
Thus every nonsingular \(S\) has at least \(\prod_{j\in B}|V_j|\ge \alpha^d n^d\) candidates.

Given a nonsingular multiplicity pattern \(P=(r_1,\ldots,r_t)\) and a transversal \(T\),
if \(T\) serves as \(T_0(y)\) in the above construction,
then \(T\) determines all canonical representatives, leaving at most \((\sigma n)^d\) choices.
Otherwise it replaces some \(b_i\) by a nonrepresentative vertex \(z\), and the pair consisting
of \(i\) and the missing part containing the selected completion vertex has at most \(t(t-1)\) choices.
The transversal \(T\) determines all other representatives.
If \(z\in U_i\), there are at most \(\gamma n\) choices for \(b_i\); if \(z,b_i\in X_i\), there are
at most \(\xi n\); and the remaining nonrepresentative vertices have at most \((\sigma n)^{d-1}\) choices.
In the only remaining case, \(z\in X_i\) and \(b_i\in U_i\). Here there are at most \(n\) choices for \(b_i\),
but nonsingularity implies \(d\ge2\), so this case contributes at most \(\sigma n^d\).
Consequently there is a constant \(C=C(t)\) such that \(T\) is a candidate for at most \(C\sigma n^d\)
nonsingular copies of pattern \(P\) by our construction.

We construct a bipartite \emph{candidate graph} \(H\) as below: its left side is the non-transversal \(K_t\)'s,
grouped by their multiplicity patterns plus one additional pattern for all singular non-transversals;
its right side is \(\mathfrak T\), there is an edge between a non-transversal \(K_t\) \(S\) and a transversal \(T\)
if and only if \(T\) is a candidate for \(S\) obtained through the above construction.
Let \(L=L(t)\) be the number of multiplicity patterns plus one.
Fix \(R\subseteq\mathfrak T\) with \(|R|\le2\). We would like non-transversals of different patterns to use
disjoint classes of candidates. For this purpose, we assign the members of \(\mathfrak T\setminus R\) independently
and uniformly to the \(L\) multiplicity patterns. Later on, for a non-transversal \(K_t\) \(S\) of pattern \(P\),
we only use candidates assigned to \(P\). Specifically, each of \(S\)'s original candidates outside \(R\) is preserved
with probability \(1/L\) independently. Delete the edges between each non-transversal \(K_t\) and
transversals not assigned to its pattern. Denote by \(\mathbf H\) the resulting bipartite graph. Then
\begin{equation*}
\mathbb E|N_{\mathbf H}(S)|=\frac{1}{L}\left|N_{H}(S)\setminus R\right|
\end{equation*}
for any non-transversal \(K_t\) \(S\). Recall that preceding analysis gives
\begin{equation*}
  |N_H(S)|\ge
  \begin{cases}
    \kappa n & S  \mathrm{\ is\ singular},\\
    \alpha^dn^d & S \mathrm{\ is\ nonsingular},
  \end{cases}
\end{equation*}
where \(d\ge1\) when \(S\) is nonsingular. Since \(R\) rules out at most \(2\) candidates, a Chernoff bound gives
\begin{equation*}
  \Pr\left(|N_{\mathbf{H}}(S)|<\frac{|N_H(S)|}{4L}\right)\le e^{-cn}
\end{equation*}
for large enough \(n\), where \(c=\min\{\kappa,\alpha^t\}/(16L)>0\). There are fewer than \(n^t\) non-transversal
\(K_t\)'s and \(n^te^{-cn}<1\) for large enough \(n\). Consequently, there is an assignment such that each
non-transversal \(K_t\) \(S\) retains at least \(\frac{1}{4L}\) of its original candidates.
Denote the resulting bipartite graph under this assignment by \(\widetilde{H}\). By our construction,
\(\widetilde{H}\) is the union of \(L\) disjoint bipartite subgraphs \(\widetilde{H}_P\), one for each pattern,
plus the isolated set \(R\). By preceding analysis, for each transversal \(T\) on the right side assigned to pattern \(P\), we have
\begin{equation*}
  |N_{\widetilde{H}_P}(T)|\le
  \begin{cases}
    t(t-1)\xi n & P  \mathrm{\ is\ the\ singular\ pattern},\\
    C\sigma n^d & P \mathrm{\ is\ nonsingular}.
  \end{cases}
\end{equation*}

Given \(t\ge 2\) and \(0<\alpha\le 1/t\), choose
\begin{equation*}
\gamma:=\min\left\{\frac13,\frac{\alpha^t}{8CL}\right\}.
\end{equation*}
Given \(\kappa>0\), choose \(\xi>0\) so small that
\begin{equation*}
\xi\le\gamma,
\qquad
t(t-1)\xi\le\frac{\kappa}{4L}.
\end{equation*}
Then on every \(\widetilde{H}_P\) we have minimum left degree at least as large as maximum right degree.
Write the two degrees as \(\delta_{\widetilde{H}_P}^{\mathrm{left}}\) and \(\Delta_{\widetilde{H}_P}^{\mathrm{right}}\) respectively.
For every set \(\mathcal S\) of left vertices in \(\widetilde{H}_P\) we have
\begin{equation*}
\delta_{\widetilde{H}_P}^{\mathrm{left}}|\mathcal S|
 \le e(\mathcal S,N(\mathcal S))
 \le \Delta_{\widetilde{H}_P}^{\mathrm{right}}|N(\mathcal S)|,
\end{equation*}
so Hall's condition holds. Combining all the resulting matchings together gives an injection
\begin{equation}\label{cov:eq:transversal-injection}
\phi_R:\mathcal K\longrightarrow\mathfrak T\setminus R,
\end{equation}
where \(\mathcal K\) is the family of all non-transversal copies of \(K_t\) and \(\phi_R(S)\) is admissible for \(S\).

Apply \eqref{cov:eq:transversal-injection} with \(R=\varnothing\).
Start with all actual transversal copies, namely \(\mathcal A\). If \(\phi_\varnothing(S)\) is missing,
add \(S\) to the cover. If \(\phi_\varnothing(S)\in\mathcal A\), replace \(\phi_\varnothing(S)\) by
the clique \(S\cup\phi_\varnothing(S)\). Since \(\phi_\varnothing\) is injective,
at most \(|\mathfrak T\setminus\mathcal A|\) cliques are added for missing transversals and
each actual transversal \(K_t\) is enlarged at most once. Therefore
\begin{equation*}
\CC_t(G)\le |\mathcal A|+|\mathfrak T\setminus\mathcal A|
 =|\mathfrak T|=\prod_{i=1}^t|V_i|.
\end{equation*}

Suppose equality holds. If a missing transversal \(T\) exists, use \(\phi_{\{T\}}\)
from \eqref{cov:eq:transversal-injection}. Then at most \(|\mathfrak T\setminus\mathcal A|-1\)
non-transversal copies are paid for by missing transversals, so the cover has fewer than \(|\mathfrak T|\) members.
Hence every transversal is an actual \(K_t\), and every pair of distinct parts is complete.

If some \(V_i\) contains an edge \(uv\), choose one vertex from every other part. Together with \(u,v\),
these vertices form a \(K_{t+1}\). Let \(T_u,T_v\) be the two transversal copies of \(K_t\) contained in this clique.
Start with \(\mathcal A\setminus\{T_u,T_v\}\), apply \eqref{cov:eq:transversal-injection} with
\(R=\left\{T_u,T_v\right\}\) and add the \(K_{t+1}\) containing \(T_u,T_v\) we just selected.
This gives a cover with at most \(|\mathfrak T|-1\) members, a contradiction.
Thus every part is independent, and \(G\) is complete \(t\)-partite.
\end{proof}

Lemma~\ref{cov:lem:edit-partition} and Theorem~\ref{cov:thm:transversal-matching} together prove Theorem~\ref{cov:thm:main}.

\begin{proof}[Proof of Theorem~\ref{cov:thm:main}]
Set \(\alpha=1/(2t)\) and let \(\gamma=\gamma(t,\alpha)\) be given by Theorem~\ref{cov:thm:transversal-matching}.
Put \(\kappa=\gamma/2\), and let \(\xi=\xi(t,\alpha,\kappa)\) and \(N_1=N(t,\alpha,\kappa)\) be supplied by
Theorem~\ref{cov:thm:transversal-matching}. Choose \(\theta>0\) so small that
\begin{equation*}
\frac{2\theta}{\gamma}\le\xi,
\qquad
2\sqrt{\theta}<\min\left\{\frac1{2t},\frac{\gamma}{2}\right\}.
\end{equation*}
Let \(\delta>0\) and \(N_2\) be supplied by Corollary~\ref{cov:cor:integer-stability} with error \(\theta\),
and choose \(n_0=n_0(t)\ge\max\{N_1,N_2\}\) so large that
\begin{equation*}
\sqrt{4\theta+\frac{t}{4n_0^2}}
 \le \min\left\{\frac1{2t},\frac{\gamma}{2}\right\}.
\end{equation*}

Let \(G\) be an \(n\)-vertex graph with \(n\ge n_0(t)\).  If \(\CC_t(G)<\CC_t(T_{n,t})-\delta n^t\),
there is nothing to prove.  Otherwise Corollary~\ref{cov:cor:integer-stability} gives \(\ed(G,T_{n,t})\le\theta n^2\).
Choose a partition \(\mathcal P=(V_1,\ldots,V_t)\) minimizing \(D_G\),
and define \(X_i,U_i,X\) as in Lemma~\ref{cov:lem:edit-partition}. The lemma and the choices above give
\begin{align*}
|V_i|&\ge\frac nt-\sqrt{4\theta n^2+\frac t4}\ge\alpha n,\\
|X|&\le\frac{2\theta}{\gamma}n\le\xi n,\\
d(u,V_i)&\le\gamma n, && \mathrm{for\ } i\in [t],\ u\in U_i,\\
d(x,V_j)&\ge\gamma n-\sqrt{4\theta n^2+\frac t4}\ge\kappa n,
 && \mathrm{for\ } i\in [t],\ x\in X_i,\ j\ne i.
\end{align*}
Theorem~\ref{cov:thm:transversal-matching} therefore applies and gives
\begin{equation*}
\CC_t(G)
 \le \prod_{i=1}^t|V_i|
 \le \CC_t(T_{n,t}).
\end{equation*}
This proves the upper bound.

Suppose now that \(\CC_t(G)=\CC_t(T_{n,t})\). Equality must hold in both inequalities above. 
Equality in Theorem~\ref{cov:thm:transversal-matching} forces \(G\) to be complete \(t\)-partite with parts \(V_1,\ldots,V_t\),
and equality in the product bound forces the part sizes to differ by at most one. Hence \(G\cong T_{n,t}\).
\end{proof}

\section*{Acknowledgements}
The author would like to thank Jie Ma for bringing this problem to his attention and for helpful discussions.

\section*{Declaration on the Use of AI}
The central proof idea in this manuscript was first generated with the assistance of GPT 5.6 Sol Ultra. The manuscript was drafted with the assistance of Codex (using GPT 5.6 Sol Ultra). All AI-assisted outputs were independently checked and refined by the author, who takes full responsibility for the correctness and originality of the paper.

\printbibliography

\end{document}